\documentclass[11pt]{amsart}
\usepackage{amsmath,amsthm,amsfonts,amssymb,times,bm}
\usepackage{latexsym}
\usepackage{mathtools}
\usepackage{mathrsfs} 

\usepackage{xcolor}
\newcommand{\red}[1]{{#1}}

\numberwithin{equation}{section}  

\DeclareMathAlphabet{\curly}{U}{rsfs}{m}{n}  

\theoremstyle{remark}
\newtheorem{remark}{Remark}
\newtheorem{example}{Example}
\theoremstyle{plain}

\newtheorem{lemma}{Lemma}[section]
\newtheorem{theorem}{Theorem}
\newtheorem{cor}{Corollary}

\newtheorem{definition}{Definition}

\newcommand{\N}{\mathbb{N}}
\newcommand{\Z}{\mathbb{Z}}
\newcommand{\Q}{\mathbb{Q}}

\newcommand{\E}{\mathbb{E}}   
\newcommand{\PR}{\mathbb{P}}  

\makeatletter
\renewcommand{\pmod}[1]{\allowbreak\mkern7mu({\operator@font mod}\,\,#1)}
\makeatother

\newcommand{\bal}{\[\begin{aligned}}
\newcommand{\eal}{\end{aligned}\]}

\newcommand{\be}{\begin{equation}}
\newcommand{\ee}{\end{equation}}

\newcommand{\ssum}[1]{\sum_{\substack{#1}}}  
\newcommand{\sprod}[1]{\prod_{\substack{#1}}}  

\newcommand{\eps}{\ensuremath{\varepsilon}}

\renewcommand{\le}{\leqslant}
\renewcommand{\leq}{\leqslant}
\renewcommand{\ge}{\geqslant}
\renewcommand{\geq}{\geqslant}

\newcommand{\fl}[1]{{\ensuremath{\left\lfloor {#1} \right\rfloor}}}
\newcommand{\cl}[1]{{\ensuremath{\left\lceil #1 \right\rceil}}}
\newcommand{\order}{\asymp}      
\renewcommand{\(}{\left(}
\renewcommand{\)}{\right)}
\newcommand{\ds}{\displaystyle}
\newcommand{\pfrac}[2]{\left(\frac{#1}{#2}\right)}  

\newcommand{\bb}{\ensuremath{\mathbf{b}}}

\renewcommand{\SS}{\mathbf{S}}

\renewcommand{\mod}{\bmod}  

\newcommand{\coloneqq}{\,:=}

\begin{document}

\title{Long gaps in sieved sets}

\author{Kevin Ford}
\address[Corresponding author]{Department of Mathematics\\ 1409 West Green Street \\ University
of Illinois at Urbana-Champaign\\ Urbana, IL 61801\\ USA}
\email{ford@math.uiuc.edu}

\author{Sergei Konyagin}
\address{Steklov Mathematical Institute\\
8 Gubkin Street\\
Moscow, 119991\\
Russia}
\email{konyagin@mi-ras.ru}

\author{James Maynard}
\address{Mathematical Institute\\
Radcliffe Observatory Quarter\\
Woodstock Road\\
Oxford OX2 6GG\\
England }
\email{james.alexander.maynard@gmail.com}

\author{Carl Pomerance}
\address{
Mathematics Department \\
Dartmouth College\\
Hanover, NH 03755, USA}
\email{carl.pomerance@dartmouth.edu}

\author{Terence Tao}
\address{Department of Mathematics, UCLA\\
405 Hilgard Ave\\
Los Angeles CA 90095\\
USA}
\email{tao@math.ucla.edu}

\thanks{KF was supported by National Science Foundation grant DMS-1501982.
JM was supported by a Clay Research Fellowship and a Fellowship of Magdalen College, Oxford.
TT was supported by a Simons Investigator grant, the James and Carol Collins Chair, the Mathematical Analysis \& Application Research Fund Endowment, and by NSF grant DMS-1266164.
Part of this work was carried out at MSRI, Berkeley during the Spring semester of 2017, supported in part by NSF grant DMS-1440140. We thank the anonymous referees for many useful suggestions.}



\begin{abstract} For each prime $p$, let $I_p \subset \Z/p\Z$ denote a collection of residue classes modulo $p$ such that
the cardinalities $|I_p|$ are bounded and about $1$ on average.
We show that for sufficiently large $x$, the sifted set $\{ n \in \Z: n \pmod{p} \not \in I_p \hbox{ for all }p \leq x\}$ contains gaps of size at least $x (\log x)^{\delta} $ where $\delta>0$ depends only on the density of primes for which $I_p\ne \emptyset$.  This improves  on the ``trivial'' bound of $\gg x$.  As a consequence, for any non-constant polynomial $f:\Z\to \Z$ with positive leading coefficient,
 the set $\{ n \leq X: f(n) \hbox{ composite}\}$ contains an interval of consecutive integers of length $\ge (\log X) (\log\log X)^{\delta}$ for sufficiently large $X$, where $\delta>0$ depends only on the degree of $f$.
 
\bigskip

{\color{red}This version of the paper incorporates a number of corrections pointed out to the authors by Mikhail Gabdullin.  The specific changes are enunciated in Appendix A. }
\end{abstract}

\date{\today}
\maketitle
%

\section{Introduction}

It is well-known that the sieve of Eratosthenes sometimes removes unusually long strings of consecutive integers,
and this implies that the sequence of primes occasionally has much
longer gaps than the average spacing.  It might be expected that similar methods would show analogous results for other sets undergoing a sieve, such as sets defined by polynomials.  For example, we know that the number of $n\le x$ with $n^2+1$
prime is $O(x/\log x)$, so an immediate corollary is that there are intervals of length $\gg\log x$
below $x$ where $n^2+1$ is composite for each $n$ in the interval. 
Can we do better?  A simple averaging argument is not useful, since the $O(x/\log x)$ bound for the count is conjecturally best possible.  In addition, there unfortunately appear to be fundamental
obstructions to adapting the
 methods used to locate large gaps in the
 Eratosthenes sieve to this situation.

In this paper we introduce a new method which 
substantially improves upon the trivial bound for
 these polynomial sets, and applies to
more general sieving situations.
We consider the set of integers remaining after applying a ``one-dimensional'' sieve, and show that this sieved set contains some unusually large gaps. To state our theorem precisely we require the following definition. The symbol $p$ always denotes a prime.

\begin{definition}[Sieving System]\label{sysdef}  A \emph{sieving system}
is a collection ${\mathcal I}$ of sets $I_p \subset \Z/p\Z$ of residue
classes modulo $p$ for each prime $p$.  Moreover, we have the following definitions.
\begin{itemize}
\item (Non-degeneracy) We say that the sieving system is
\emph{non-degenerate} if $|I_p|\le p-1$ for all $p$. 
\item ($B$-Boundedness) Given $B>0$, we say that the sieving system is \emph{$B$-bounded} if
\begin{equation}\label{ip-bound}
|I_p|\le B~\hbox{ for all primes }~p.\end{equation}
\item (One-dimensionality) We say that the sieving system is \emph{one-dimensional} if we have the weighted Mertens-type product estimate
\be\label{Mertens}
\prod_{p\le x} \(1-\frac{|I_p|}{p}\) \sim \frac{C_1}{\log x}  \qquad (x\to\infty),
\ee
for some constant $C_1>0$. 
\item ($\rho$-supportedness) Given $\rho>0$, we say that the sieving system system is \emph{$\rho$-supported} if
the density of primes with $|I_p|\ge 1$ equals $\rho$, that is,
\be\label{rho}
\lim_{x\to\infty} \frac{|\{p\le x : |I_p|\ge 1\}|}{x/\log x}=\rho.
\ee
\end{itemize}

\end{definition}
Roughly speaking, a ``sieving system'' which is non-degenerate, $B$-bounded, 1-dimensional and $\rho$-supported  specifies certain residue classes for each prime $p$, such that there is roughly 1 residue class per prime on average, and if we remove all integers in these residue classes the resulting set is not too erratic.  

Given such a sieving system $\mathcal{I}$, our main object of study is the \emph{sifted set}
$$ S_x =S_x(\mathcal{I}) \coloneqq \Z\setminus \bigcup_{p\le x} I_p,
$$
of integers which are not contained in any of the residue classes specified by the $I_p$ for $p\le x$. If $|I_p|=p$ for some $p\le x$ (the degenerate case), then clearly  $S_x$ is empty.  Otherwise, $S_x$ is a $P(x)$-periodic set with density $\sigma(x)$, where $P(x)$ and $\sigma(x)$ are defined as
\[
P(x):=\prod_{\substack{p\le x\\I_p\ne\emptyset}} p, \qquad
\sigma(x) \coloneqq \prod_{p\le x} \left( 1 - \frac{|I_p|}{p} \right).
\]
We  also note that $S_x\supseteq S_y$ if $x\le y$. With this set-up we can now state our main theorem.

\begin{theorem}[Main theorem]\label{main}  Let ${\mathcal I}$ be a non-degenerate,
$B$-bounded, one-dimensional, $\rho$-supported sieving system
with $\rho>0$.  Define
\be\label{Crho}
\red{C(\rho):=\sup\Bigl\{\delta\in(0,1/2):\,\frac{6\cdot 10^{2\delta}}{\log(1/(2\delta))}<\rho\Bigr\}}.
\ee
The sifted set $S_x$ contains a gap of length at least $x (\log x)^{C(\rho)-o(1)}$, where the rate of decay of the
$o(1)$ bound depends on $\mathcal{I}$.
\red{Moreover, $C(\rho) > e^{-1-6/\rho}$.}
\end{theorem}

\begin{remark}
We note that since $\mathcal{I}$ is one-dimensional, we must have that
$$
\rho \ge \frac{1}{B}.
$$
(So, for example, the positivity of $\rho$ follows from the property that $\mathcal{I}$ is $B$-bounded.)
The value of $C_1$ in \eqref{Mertens}, which has no importance for our 
arguments, depends on the
behavior of $|I_p|$ for small $p$, and can have great variation.

Condition \eqref{rho} is used primarily to construct large sets of primes with $I_p\ne\emptyset$ in very short intervals, see
\eqref{QH} below.  It is possible to weaken \eqref{rho}
further, e.g. so that \eqref{QH} holds for most scales $H$
instead of all $H$, however this would further complicate 
our argument.  All of the canonical examples satisfy \eqref{rho}.
\end{remark}

There is a straightforward argument that shows that $S_x$ must have gaps of length $\gg x$, for $x$ sufficiently large in terms of $\mathcal{I}$ --- see Remark \ref{TrivialBound} below. Theorem \ref{main} improves over this bound by a positive power of $\log{x}$, and it is the fact that we get a non-trivial result in this level of generality which is the main point of the Theorem. It is likely that with more effort one could improve the bounds on the constant $C(\rho)$; our main interest is that this is an explicit positive constant depending only on $\rho$. We now demonstrate applications of the theorem via several examples.

\begin{example}[Gaps between primes]\label{Eratosthenes} The ``Eratosthenes'' sieving system is the system with $I_p=\{0\}$ for all $p$, and it
 is non-degenerate, 1-bounded, one-dimensional and $1$-supported.  We have
\begin{equation}\label{pxx}
\{\sqrt{X} < p \le X:p\text{ prime} \} = S_{\sqrt{X}} \cap (\sqrt{X}, X].
\end{equation}
Since $S_x\supseteq S_{\sqrt{X}}$ if $x\le \sqrt{X}$, any large gap in $S_x$ implies a large gap in $S_{\sqrt{X}}$. Since $S_x$ is $P(x)$-periodic, if it contains a large gap then it must contain one in the interval $[X,3X]$ if $P(x)\le X$. Thus, choosing $x\approx \log{X}$ maximally such that $P(x)\le X$, we see that Theorem \ref{main} implies that there is a prime gap in $[X,3X]$ of size 
\[
\red{\gg (\log X )
(\log\log X)^{C(1)-o(1)} \gg (\log X)(\log\log X)^{1/835},}
\]
\red{on numerically calculating that $C(1)>1/835$} (the limit of our type of method appears to be an exponent $1/e$; see Remark \ref{rem:limit} in Section \ref{sec:concentration}). This is stronger than the trivial bound of $(1+o(1))\log{X}$, which is immediate from the Prime Number Theorem, but is worse than the current best bounds for this problem. Indeed, the problem of finding large gaps between consecutive primes has a long history,
and it is currently known that gaps of size
\be\label{primegap}
\gg \log X\,\frac{\log\log X\,\log\log\log\log X}{\log\log\log X}
\ee
 exist below $X$ if $X$ is large enough, a recent result of Ford, Green, Konyagin, Maynard,
and Tao\ \cite{FGKMT}. The key interest is that Theorem \ref{main} applies to much more general sieving situations, to which it appears difficult to adapt the previous techniques, and gives a different method of proof to these previous results. We will discuss the reasons for this in detail below.
\end{example}


\begin{example}[Gaps between prime values of polynomials]
Given a polynomial $f:\Z\to\Z$ of degree $d \geq 1$,
consider the system $\mathcal{I}$ with $I_p=\emptyset$ for $p\le d$ and 
$$ I_p \coloneqq \{ n \in \Z/p\Z: f(n)\equiv 0\pmod{p} \} $$
for $p>d$.
The polynomial need not have integer coefficients, e.g.
$f(n)=\frac{n^7-n+7}{7}$ satisfies the hypotheses of Theorem \ref{thmpoly}. By P\'olya's theorem \cite{Polya},
$f$ is integer valued at integers if and only if $f$ 
has the form $f(x)=\sum_{j=0}^d a_j \binom{x}{j}$ with every $a_j\in \Z$.  In particular, $d! f(y) \in \Z[y]$ and
thus the sieving system is well-defined.

By Lagrange's theorem, $|I_p| \le d < p $ for all $p>d$,
and hence the system is non-degenerate and $d$-bounded.
 For irreducible $f$,
the one-dimensionality \eqref{Mertens} with strong error term
follow quickly from Landau's Prime
Ideal Theorem \cite{Landau} (see also \cite[pp. 35--36]{CoMu}), while \eqref{rho}, the $\rho$-supportedness of the system with $\rho\ge 1/d$,
 follows from the Chebotarev Density
Theorem \cite{Chebotarev} (see also \cite{lagodl}).
   As a variant of \eqref{pxx}, we observe that
\begin{equation*}
 \{ n \in \N: f(n)>x, f(n) \text{ prime} \} \subset S_x
\end{equation*}
for any $x>1$.
Now set $x  \coloneqq \frac{1}{2} \log X$.  By Theorem~\ref{main}, the set
$S_x$ contains a gap of length $\gg (\log X) (\log\log X)^{C(1/d)-o(1)} $.
The period of this set, $P(x)$, is $X^{1/2+o(1)}$ by the Prime Number Theorem.  Thus, this set contains such a long gap inside the interval $[X/2,X]$.   Assuming that $f$ has a positive leading
coefficient and that $X$ is large, on
the interval $[X/2,X]$ we have $f(n)>x$, and so $f(n)$ is composite for every $n \in [X/2,X] \backslash S_x$.
We thus obtain the following.

\begin{cor}\label{thmpoly}  Let $f
:\Z\to\Z
$ be a polynomial of degree $d \geq 1$ with positive leading term.  Then for sufficiently large $X$, there is a string of consecutive natural numbers $n \in [1,X]$ of length $\ge (\log X) (\log\log X)^{C(1/d)-o(1)}$ for which $f(n)$ is composite, \red{where 
$C(1/d)>e^{-(6d+1)}$} is the constant of Theorem \ref{main}.
\end{cor}

Note that Corollary \ref{thmpoly} includes the trivial ``degenerate'' cases, when
either $f$ is reducible, or there is some prime $p$ with $|I_p|=p$, since then essentially all values of $f$ are composite.

When $f$ is irreducible, has degree two or greater, and the sieving system corresponding to $f$ is
non-degenerate, it is still an open conjecture (of Bunyakovsky \cite{bun}) that there are infinitely many integers $n$ for which $f(n)$ is prime. Moreover it is believed (see the conjecture of Bateman and Horn \cite{BatemanHorn}) that the density of these prime values on $[X/2,X]$ is $\asymp_f 1/\log{X}$, and so the gaps of Corollary \ref{thmpoly} would be unusually large compared to the average gap of size $\asymp_f \log{X}$. We do not address these conjectures at all in this paper. Of course, in the unlikely event that Bunyakovsky's conjecture was false and there were only finitely many prime values of $f$, Corollary \ref{thmpoly} would be much weaker than the truth.

\begin{remark}
Let $G$ be the Galois group of $f$, realized canonically as a subgroup of the symmetric group $\mathfrak{S}_d$. By the Chebotarev Density Theorem \cite{Chebotarev} (see also \cite{lagodl}), we may take $\rho$ equal to the proportion of elements of $G$ with at least one fixed point, which lies in $[\frac1{d},1)$.
We have $\rho=1/d$ 
for many polynomials, e.g. $x^{2^k}+1$, but $\rho$ is much larger generically.  It is known
since van der Waerden \cite{vdW} that a random irreducible polynomial of degree $d$
will have Galois group $\mathfrak{S}_d$ with high probability\footnote{Specifically, fix the degree $d$ and let the
coefficients of $f$ be chosen randomly and uniformly from
$[-N,N] \cap \Z$.  Then, as $N\to \infty$, the probability
that $f$ is irreducible and has Galois group $\mathfrak{S}_d$
tends to 1.}.
In this case $\rho$ is the proportion of elements of $\mathfrak{S}_d$ with a fixed point.
  This is the classical derangement problem, and we have for such polynomials
\[
\rho = \rho_d := \sum_{k=1}^d \frac{(-1)^{k+1}}{k!}.
\]
In particular, $\rho_d \ge 1/2$, $\rho_d \ge \frac58$ for $d\ge 3$ and $\lim_{d\to\infty} \rho_d = 1-1/e$.  A calculation reveals that
\be
\red{C(1/2)>\frac{1}{325565}}.\label{cBound}
\ee
Since $C(\rho)$ is increasing with $\rho$, we thus have the following corollary.
\end{remark}
\end{example}

\begin{cor}\label{generic-polynomial}
 Let $f:\Z\to\Z
$ be a polynomial of degree $d \geq 2$ with positive leading term, irreducible over $\Q$, and with full Galois group $\mathfrak{S}_d$.  Then for all sufficiently large $X$, there is a string of consecutive natural numbers $n \in [1,X]$ \red{of length $\ge \log X (\log\log X)^{1/325565}$} for which $f(n)$ is composite.
\end{cor}

\begin{example}
A simple example to keep in mind is $f(n) = n^2+1$.  In this case, $I_2=\{1\}$, $I_p=\emptyset$ is empty for $p \equiv 3\pmod{4}$, and $I_p = \{\iota_p, - \iota_p\}$ for $p \equiv 1\pmod{4}$, where $\iota_p \in \Z/p\Z$ is one of the square roots of $-1$.  Here one can use the Prime Number Theorem in arithmetic progressions rather than the Prime Ideal theorem to establish one-dimensionality and the $\rho$-supportedness with $\rho=1/2$.
For this example (and for any quadratic polynomial),
Theorem \ref{main} implies the existence of
consecutive composite strings \red{of length $\gg (\log X)(\log\log X)^{C(1/2)-o(1)} \gg (\log X)(\log\log X)^{1/325565}$} (using \eqref{cBound} again).  It is certain that further
numerical improvements are possible.
\end{example}

Theorem \ref{main} has another application, to a problem on the coprimality
of consecutive values of polynomials.

\begin{cor}\label{no_coprime} Let $f:\Z\to\Z
$ be a non-constant polynomial.
Then there exists an integer $G_f\ge2$ such that for any integer $k\ge G_f$
there are infinitely many integers $n\ge0$ with the property that
none of the numbers $f(n+1),\dots,f(n+k)$ are coprime to all the others.
\end{cor}

For linear polynomials the result of the corollary is well-known, and not difficult to prove;
for quadratic and cubic polynomials
 in $\Z[x]$, the result was only proven recently by Sanna and Szikszai \cite{SS}.  The remaining cases of polynomials of degree four and higher appears to be new.

\begin{proof}
Let $d=\deg f$.  Then $d!f(y)\in\Z[y]$.  
Let $f_0(y)\in\Z[y]$ be a primitive irreducible
factor of $d!
f(y)$.  If $p>d$ is a prime and $p\mid f_0(m)$ for some integer $m$,
then $p\mid f(m)$.  So it will suffice to consider the case that $f$ is irreducible and
show in this case that for all large $k$ there are infinitely many $n\ge0$ such that
for each $i\in\{1,\dots,k\}$ there is some $j\in\{1,\dots,k\}$ with $j\ne i$ and
$\gcd(f(n+i),f(n+j))$ divisible by some prime $>d$.

Again, we consider the system $\mathcal I$ defined by $I_p=\emptyset$ for $p\le d$ and for $p>d$ we take
$$ I_p \coloneqq \{ n \in \Z/p\Z: f(n)\equiv 0\pmod{p} \}.$$
By Theorem \ref{main}, for all large numbers $x$ the set
$S_x$ contains a gap of length $\ge k=\lfloor2x\rfloor$.
Thus, there are infinitely many $n$ such that each $f(n+1),\dots,f(n+k)$ has a prime
factor $p$ with $d<p\le x$.  For each $i\in\{1,\dots,k\}$, take a prime factor
$p$ of $f(n+i)$ with $d<p\le x$.  Since $k=\lfloor2x\rfloor$, $p\le x$ and $I_p\ne\emptyset$, it must be
that $p$ divides at least
two terms of the sequence $f(n+1),\dots,f(n+k)$, thus proving the assertion.
\end{proof}

\begin{remark} Our proof of  Corollary \ref{no_coprime} above requires only a very weak version of
 Theorem \ref{main}.  It is not clear, however, that
a trivial argument of the type presented below in
Remark \ref{TrivialBound} can yield a gap of size at least $2x$  when the degree of $f$ is large.
\end{remark}

\begin{remark}
The conclusion of Theorem \ref{main} is equivalent to
the existence, for any $\delta<C(\rho)$, of some
$b\in \Z/P(x)\Z$ with
\[
(S_x + b)\cap [1, x (\log x)^{\delta} ] = \emptyset,
\]
provided $x$ is sufficiently large in terms of $\delta$.
Here $S_x + b := \{s+b : s\in S_x\}$.
\end{remark}

\begin{remark}\label{TrivialBound}
 The conclusion of Theorem \ref{main} should be compared with the ``trivial'' bound: there is a constant $c'>0$
such that for each sufficiently large $x$, there is some integer $b$ with
\be\label{trivial}
 (S_x+b) \cap [1, c'x ] = \emptyset.
 \ee
We now sketch the proof of \eqref{trivial}. Firstly, we see that we may assume that $x$ is large. Then by \eqref{Mertens}
it follows that there is some $b$ modulo $P(x/2)$ for which $\mathcal{A} := (S_{x/2}+b) \cap [1, \frac{\rho x}{8C_1} ]$
satisfies $|\mathcal{A}| \le \frac{\rho x}{4\log x}$.
On the other hand, by \eqref{rho}
 for any fixed $\eps>0$ we have
\be\label{largeIqge1}
\# \{x/2 <q \le x : |I_q| \ge 1 \} \ge \(\frac{\rho}{2}-\eps\)\frac{x}{\log x} 
\ee
for large $x$.
Hence, we may perform a ``clean up stage'' in which we pair up each element $a\in \mathcal{A}$ with a unique prime
$q=q_a\in (x/2,x]$ for which $|I_q|\ge 1$.  For each such pair $a,q_a$
let $v_a \in I_{q_a}$ and suppose that $b\equiv a-v_a \pmod{q}$.
It follows that
$(S_x + b) \cap [1, \frac{\rho x}{8C_1} ] = \emptyset$, proving \eqref{trivial}.
\end{remark}
\begin{remark}
The hypothesis \eqref{ip-bound} is an important assumption in our treatment of certain error terms; see Lemma \ref{correlation} below.  It is possible to relax this hypothesis with more sophisticated arguments, and several steps of the argument could be established with slightly weaker assumptions.

The formula \eqref{Mertens} say that
$|I_p|$ has average 1 in a weak sense, and is similar to
the usual condition defining a 
\emph{one-dimensional} sieve (see e.g. \cite[Sections 5.5, 6.7]{FI}).   Most of our arguments have counterparts if the one-dimensional hypothesis \eqref{Mertens} is replaced by another dimension, but in those cases the bounds we could obtain were inferior to what could be obtained by the ``trivial'' argument; see for instance Remark \ref{romo} below.
\end{remark}

\subsection{Comparisons of methods}

Recall from Example \ref{Eratosthenes} that for the Eratosthenes sieving system $I_p=\{0\}$, previous methods were able to deduce stronger variants of Theorem \ref{main}. We now explain why these methods appear difficult to adapt to more general sieving systems.


In the Eratosthenes sieving system
 it is clear that $S_x$ avoids the interval $[2,x]$, which already gives the ``trivial'' lower bound $j(P(x)) \ge x-2$.  All of the improvements to this bound in previous literature (including those in \cite{FGKMT}) rely on a variant of the following observation: if $x \geq z \geq 2$, then the sifted set 
\be\label{Szx}
S_{z,x} = \N \setminus \bigcup_{z<p\le x} I_p,
\ee 
when restricted to the interval $[1,y)$ with $y$ slightly larger than $x$, only consists of numbers of the form $a$ or $a p$, where $p$ is a prime in $(x,y]$, and $a$ is \emph{$z$-smooth} (or \emph{$z$-friable}), which means that no prime factor of $a$ exceeds $z$. Moreover, $z$-smooth numbers are much rarer than one would expect from naive sieving heuristics (if $z$ is suitably small), but numbers of the form $a p$ must have $a$ less than $y/x$, which is also a rare factorization (if $y$ is only slightly larger than $x$). Thus the number of elements of $S_{z,x}$ in $[1,y)$ is unusually small. It is the fact that we can identify this interval containing unusually few integers after sieving by the ``medium-sized'' primes which is the key ingredient allowing one to improve on the trivial bound.

The most recent works on this problem then try to show as efficiently as possible that one choose $b$ (a multiple of $\prod_{z<p\le x}p$) such that $(b+S_{2x})\cap[1,y)=\emptyset$, and so we can sieve out out these few remaining elements of $[1,y)$. This then implies the existence of a large gap of size $y$ in $S_{2x}$. However, if we did not already know that there were few elements in $[1,y)$, then these methods would not produce a non-trivial bound. 


Unfortunately, when considering the more general sieving systems of Definition \ref{sysdef} in which the cardinalities $|I_p|$ are allowed to vanish for many primes $p$,  bounds for smooth numbers cannot be used to show that $S_{z,x}$ contains an interval with unusually few elements. Without this crucial step the existing methods
only yield the trivial lower bound of $\gg x$ for the gap size. Moreover, for a general sieving system which is $\rho$-supported with $\rho<1$, we expect that \emph{no} such reasonably long interval containing so few elements will exist in $S_{z,x}$, meaning that this feature is genuinely unique to the Eratosthenes sieving system.

We overcome this obstacle by using a rather different method. Rather than attempting to do unusually well with the medium sized primes $p< x/(\log{x})^{1/2}$, we instead will make random choices, and only obtain results comparable to the trivial bound. We obtain an improvement over the trivial bound by working harder with the larger primes $p\in[x/(\log{x})^{1/2},x]$, showing that for each of these larger primes we can actually remove more elements that one would typically expect by choosing the residue class carefully. In order to make sure these choices do not interfere with each other too much, we make the choices randomly in several stages, where the random choice is conditional on the previous stages.

The basic idea is similar to how recent papers (e.g. \cite{FGKMT}) have exploited the large primes to sieve efficiently. In those papers one needed estimates of tuples of linear forms taking many prime values frequently, here we just need to show the existence of suitable residue classes containing unusually many unsieved integers. However, in the new set-up we require rather stronger quantitative bounds than is available for tuples of prime values -  our method would completely fail to improve over the trivial bound if we were not able to obtain close-to-optimal quantitative results. This strategy is discussed in more detail in the next section.


\begin{remark}\label{romo}  Unfortunately our methods only seem to give good results in the one-dimensional case.  Consider for instance the set $\{ n \in {\mathcal P}: n+2 \in {\mathcal P}\}$ of (the lower) twin primes.  This corresponds to a two-dimensional system in which $I_p = \{ 0 \pmod{p}, 2\pmod{p}\}$ for all primes $p$.  The ``trivial'' bound coming from these methods would give a bound of $\gg \log X \log\log X$ for the largest gap between lower twin primes up to $X$ (or between the largest such twin prime and $X$), and one could possibly hope to improve this bound by a small power of $\log\log X$ using a variant of the methods in this paper. However, a sieve upper bound
(e.g., \cite[Cor. 2.4.1]{HR})
combined with the pigeonhole principle already gives a bound of $\gg \log^2 X$ in this case.
\end{remark}

\subsection{Notation}\label{notation-sec}
From now on, we shall fix a non-degenerate, $B$-bounded, one-dimensional, $\rho$-supported sieving system $\mathcal{I}$.

We use $X \ll Y, Y \gg X$, or $X = O(Y)$ to denote the estimate $|X| \leq CY$ for some constant $C>0$, and write $X \asymp Y$ for $X \ll Y \ll X$.
Throughout the remainder of the paper,
all implied constants in $O(\cdot)$ and related order estimates may depend
on $\mathcal{I}$, in particular on the constants $B,\rho,C_1$. Moreover, implied constants will also be allowed to depend on quantities $\delta,M,K$, and $\xi$ which we specify in the next section.
We also assume that the quantity $x$
is sufficiently large in terms of all of these parameters.

 The notation $X = o(Y)$ as $x \to \infty$
means $\lim_{x\to\infty} X/Y= 0$ (holding other parameters fixed).

If $S$ is a statement, we use $1_S$ to denote its indicator, thus $1_S=1$ when $S$ is true and $1_S=0$ when $S$ is false.

We will rely on probabilistic methods in this paper.  Boldface symbols such as $\mathbf{n}$, $\mathbf{S}$, $\bm{\lambda}$, etc. denote random variables (which may be real numbers, random sets, random functions, etc.).  Most of these random variables will be discrete (in fact they will only take on finitely many values), so that we may ignore any technical issues of measurability; however it will be convenient to use some continuous random variables in the appendix.  We use $\PR(E)$ to denote the probability of a random event $E$, and $\E \mathbf{X}$ to denote the expectation of the random (real-valued) variable $\mathbf{X}$.
%

Unless specified, all sums are over the natural numbers. An exception is made for sums over the variables $p$ or $q$ (as well as variants such as $p_1$, $p_2$, etc.), which will always denote primes.
%

\red{\textbf{Acknowledgement.}
The authors thank Mikhail Gabdullin for informing us of the error in the
proof of Theorem 2 (iii), which necessitates taking $M>6$ rather than
$M>4+\delta$ as claimed, as well as other more minor errors.  Corrections from
the final, published version, are highlighted in red.}

%
%
\section{Outline}
%

In this section we describe the high-level strategy of proof, and perform two initial reductions on the problem, ultimately leaving one with the task of proving Theorem \ref{second} below.  Recall the definition \eqref{Szx} of the sifted set
$S_{z,x}$ and define related quantities
\[
P(z,x)  := \sprod{z<p\le x \\ |I_p|\ge 1}p, \qquad
\sigma(z,x) := \prod_{z<p\le x} \(1-\frac{|I_p|}{p}\). 
\]
Suppose $x$ is large (think of $x\to\infty$), and define
\begin{equation}\label{y-def}
 y := \lceil x (\log x)^{\delta} \rceil
\end{equation}
and
\be\label{z-def}
z \coloneqq \frac{y\log\log x}{(\log x)^{1/2}},
\ee
where  $\delta\in(0,1/2)$ satisfies $\delta<C(\rho)$. We recall from \eqref{Crho} that this is equivalent to
\be\label{delta1}
\red{\frac{ 6\cdot 10^{2\delta}}{\log(1/(2\delta))} < \rho},
\ee
which is a condition that will arise naturally in the proof.
Our goal is to show that $(S_x+b) \cap [1,y] = \emptyset$ for some $b$ and to accomplish this with maximal $\delta$ such that \eqref{delta1} holds.  
For a general $\rho$, it is easy to see that
\[
\red{C(\rho) > e^{-1-6/\rho} \qquad (0 < \rho\le 1)},
\]
establishing the final claim in Theorem \ref{main}.
\red{Incidentally,
$C(\rho) \sim \frac12 e^{-6/\rho}$ as $\rho\to 0^+$.}

In the course of the proof, we will introduce three additional parameters: \red{$M$ is a fixed number slightly larger than 6},  $\xi$ is a real number slightly large than 1, and $K$ is a very large integer; we will eventually take \red{$M\to 6^+$}, $\xi\to 1^+$
and $K\to \infty$.  
We adopt the convention that constants implied by 
$O(\cdot)$ and $\ll$ bounds may depend on $\delta,M,K,\xi$, in addition to the parameters defining $\mathcal{I}$, that is $\rho$, $B$, $C_1$.  Dependence on any other parameter will be stated
explicitly.
 
We observe that 
a linear shift of any single set $I_p$ (that is,
replacing $I_p$ by $c+I_p$ for some integer $c$) does not affect the structure
of $S_x$.  Thus, the same is true for linear shifts
(depending on $p$) for any finite set of primes $p$.  In particular, we may shift the sets $I_p$ so that all nonempty sets $I_p$
contain the zero element, without changing the structure of $S_x$.
Therefore, we may
assume without loss of generality that $0\in I_p$ whenever $I_p$ is nonempty.  By the Chinese Remainder Theorem, 
we may select $b$ by choosing residue classes for $b$ modulo primes $p\le x$.
\subsection{Basic Strategy}
For $x$ large enough we have
$$ 1 \leq z \leq x/2 \leq x \leq y.$$

We will select the parameter $b$ modulo the primes $p\le x$ in three stages:
\begin{enumerate}
\item (Uniform random stage) First, we choose $b$ modulo $P(z)$ uniformly at random; equivalently, for each
prime $p\le z$ with $|I_p|\ge1$, we choose $b\mod p$ randomly with uniform probability, independently for each $p$.
\item (Greedy stage) Secondly, choose $b$ modulo $P(z,x/2)$ randomly, but dependent on the choice of
$b$ modulo $P(z)$.   A bit more precisely, for each prime $q\in (z,x/2]$ with $|I_q|\ge 1$,
we will select $b\equiv b_q\pmod{q}$
so that $\{b_q+kq : k\in \Z\} \cap [1,y]$
knocks out nearly as many elements of the random set $(S_z+b)\cap [1,y]$ as possible.  Note that we are focusing only on those residues sifted by the element
$0\in I_q$, and ignoring all other possible elements of $I_q$.  This simplifies our analysis considerably, but has the effect of making $C(\rho)$ decay rapidly as $\rho\to 0$.
\item (Clean up stage) Thirdly, we choose $b$ modulo primes $q\in (x/2,x]$ to ensure that
the remaining elements $m\in (S_{x/2}+b) \cap [1,y]$ do not lie in $(S_x+b)\cap [1,y]$ by matching a unique prime $q=q(m)$ with $|I_q|\ge 1$ to each element $m$ and setting $b\equiv m \pmod{q}$. (Again we use the single element $0\in I_q$.  Such a clean up stage is standard in this subject, for instance it was already used in the proof of \eqref{trivial}.)
\end{enumerate}
We then wish to show that there is a positive probability that the above random sieving procedure has $(S_x+b)\cap[1,y]=\emptyset$, which then clearly implies that there is a choice of $b$ such that this is the case, giving Theorem \ref{main}. It is the second sieving stage above which is the key new content of this paper.

Following the argument used to show \eqref{trivial},
and using \eqref{largeIqge1}, we can successfully show that there exists a $b'$ such that 
$(S_x+b') \cap [1,y] = \emptyset$ after Stage (3)
provided that we have suitably few elements after Stage (2). By \eqref{largeIqge1} (a consequence of our hypothesis \eqref{rho}), it is sufficient to show that there is a $b$ such that
\be\label{first_reduction}
|(S_{x/2}+b) \cap [1,y]| \leq \(\frac{\rho}{2}-\eps\) \frac{x}{\log x}.
\ee

After Stage (1), from \eqref{Mertens} we see that the expected size of
$|(S_z+b)\cap [1,y]|$ is $\sim \sigma(z)y \asymp \frac{y}{\log z} \sim \frac{y}{\log x}$.  A random, uniform choice of $b$ modulo
primes $q\in (z,x/2]$ would only reduce the residual set
by a factor $\prod_{z<p\le x/2} (1-|I_p|/p) \sim 1$
and would lead to a version of Theorem \ref{main} with a gap of size $\order x$.
Instead, we use a greedy algorithm to select $b\equiv b_q\pmod{q}$.
 By \eqref{y-def} and \eqref{z-def}, the set
 $(b_q \mod q) \cap [1,y]$ has size about $H:=y/q$,
 with $(\log x)^{\delta} \ll H \ll (\log x)^{1/2}/\log\log x$. 
By considering the initial portion $(S_{H^M}+b)$ (for some fixed $M>1$) of the sieving process, one can see (e.g. using the large sieve \cite[Lemma 7.5 and Cor. 9.9]{FI} or Selberg's sieve \cite[Sec. 1.2]{Ho}) that the size of the intersection $(b_q \mod q) \cap (S_{H^M}+b) \cap [1,y]$ must be somewhat smaller, namely of size
$$ \ll \sigma(H) H \order \frac{H}{\log H}$$
by \eqref{Mertens}.
We will show that there are choices for the residues $b_q$ so that no further size reduction occurs when one sieves up to $z$ instead of $H^{M}$, namely that
\be\label{good_bq}
(S_{H^M}+b)\cap (b_q \mod q) \cap [1,y] = (S_z+b) \cap (b_q \mod q) \cap [1,y].
\ee
Heuristically, each individual choice of $b_q$ is expected to obey \eqref{good_bq} with probability
 roughly \[
 \sigma(H^{M},z)^{H\sigma(H) },\]
  but with our choice of parameters
and \eqref{Mertens}, this quantity is substantially larger than $1/q$, and so there should be many possibilities for $b_q$ for each $q$.
By contrast, for most choices of  $b_q$,
the ratio of the left and right sides of \eqref{good_bq}
is about $\sigma(H^M,z)=\prod_{H^M<p\le z} (1-|I_p|/p) \sim \frac{\log H^M}{\log z}$, which is very small.

\begin{remark}
A simple way to perform the greedy stage would be to choose the $b_q$ independently from one another for each $q$, conditional only on the first stage. One would then expect that
that we will achieve \eqref{first_reduction} if  $y=x (\log\log x)^{\rho-\eps}$ instead of \eqref{y-def}. This would give a non-trivial result which is weaker than Theorem \ref{main}.
Indeed, imagine we had instead defined $z := x/J$ and $y :=L x$, where $J$ and $L$ lie in $\bigl[100, (\log x)^{1/3}\bigr]$.
After Stage (1), we are left with a set $\mathcal{R}$ of approximately $y/\log x  = L x/\log x$ integers.
 The goal is to choose $b=b_q$ for primes
$q\in(z,x/2]$ with nonempty $I_q$ so that $b \mod q$ knocks out $\approx (y/q)/(\log(y/q))$ elements of
$\mathcal{R}$. For this to be possible, we must have $\sigma(H^{M},z)^{H\sigma(H)} \geq 1/q$
for all $H \leq y/z = JL$, but this is true on account of
$JL \leq (\log x)^{2/3}$.
 Assuming independence of all these steps (that is, for different $q$),
the residual set after the greedy sieving has size
$$\lesssim |\mathcal{R}| \sprod{x/J<q\le x/2 \\ I_q\ne \emptyset} \(1-
\frac{(y/q)/\log(y/q)}{|\mathcal{R}|} \) \approx
\frac{Lx}{\log x} \sprod{x/J<q\le x/2 \\ I_q\ne\emptyset}
\left( 1 - \frac{\log x}{q\log(y/q)} \right).$$
By the Prime Number Theorem and \eqref{rho},
\begin{align*}
\ssum{x/J<q\leq x/2 \\ I_q\ne\emptyset} \frac{\log x}{q\log(y/q)} = \rho
\int_{x/J}^{x/2} \frac{dt}{t\log(y/t)} + O(1) 
&=\rho \log\pfrac{\log JL}{\log L} + O(1),
\end{align*}
and thus the residual set has size $O(\frac{L\log L}{\log JL} \frac{x}{\log x})$.   Taking
$J = (\log x)^{1/3}$ and $L = (\log\log x)^{\rho-\eps}$, the residual set has size at most $o(x/\log{x})\le (\rho/2-\eps)\frac{x}{\log x}$,
which gives \eqref{first_reduction}, and so we're done.
\end{remark}

\subsection{The Greedy Stage: Further details}
To successfully show \eqref{first_reduction} with $y$ as large as $x(\log{x})^\delta$, we use a hypergraph covering lemma of Pippenger-Spencer type introduced
in \cite{FGKMT}.  This allows us to select residues $b_q$ such that the sets
\[
(S_{H^M}+b) \cap (b_q\mod q) \cap [1,y]
\]
are nearly disjoint.

It is convenient to separately consider the primes $q\in(z,x/2]$ in finer-than-dyadic blocks.  Fix a real number $\xi>1$ (which we will eventually take very close to 1) and define
\be\label{H-def}
{\mathfrak H} \coloneqq \left\{ H \in\{1, \xi,\xi^2,\dots\}: \frac{2 y}{x} \le H \le \frac{y}{\xi z}\right\}
\ee
be the set of relevant scales $H$; we will consider those primes $q$ in $(y/(\xi H),y/H]$ separately for each $H\in \mathfrak{H}$, noting that
$\cup_{H\in \mathfrak{H}} (\frac{y}{\xi H},\frac{y}{H}]$,
is a subinterval of $(z,x/2]$. 
By \eqref{z-def} and \eqref{y-def} for $H\in \mathfrak{H}$ we have
\be\label{H-bounds}
2 (\log x)^{\delta} \le H \le \frac{(\log x)^{1/2}}{\log\log x}.
\ee
For each $h\in \mathfrak{H}$, let ${\mathcal Q}_H$ be the set
of primes $q \in (y/(\xi H),y/H]$ with $|I_q| \geq 1$. 
From \eqref{rho}, we have
\begin{equation}\label{QH}
|{\mathcal Q}_H| \sim \rho(1-1/\xi)\frac{y}{H \log x}.
\end{equation}
Let
\[
\mathcal{Q} = \bigcup_{H\in\mathfrak H} \mathcal{Q}_H.
\]
For $q\in \mathcal{Q}$, let $H_q$ be the unique element of
$\mathfrak{H}$ such that
\be\label{Hq}
\frac{y}{\xi H_q} < q \le \frac{y}{H_q}.
\ee

Now fix a real number $M$ satisfying
\be\label{M}
\red{6< M \le 7}.
\ee
With $H$ fixed, we will examine separately the effect of the sieving by primes in $[2,H^{M}]$ and by the primes in $(H^{M},z]$.
We denote by $\bb$ a random residue class from $\Z/P\Z$, chosen with uniform probability, where we adopt the abbreviations
\[
P = P(z), \quad \sigma = \sigma(z), \quad \SS = S_z + \bb
\]
as well as the projections
\be\label{abbr-1}
P_1 = P(H^M), \quad \sigma_1 = \sigma(H^M), \quad
\bb_1 \equiv \bb \pmod{P_1}, \quad \SS_1 = S_{H^M} + \bb_1
\ee
and
\be\label{abbr-2}
P_2 = P(H^{M},z), \quad \sigma_2 = \sigma(H^M,z), \quad
\bb_2 \equiv \bb \pmod{P_2}, \quad \SS_2 = S_{H^M,z}+ \bb_2
\ee
with the convention that $\bb_1 \in \Z/P_1\Z$ and $\bb_2 \in \Z/P_2\Z$.
Thus, $\bb_1$ and $\bb_2$ are each uniformly distributed, are independent of each other, and likewise $\SS_1$ and $\SS_2$ are independent.
We also have the obvious relations
\[
P=P_1 P_2, \quad \sigma = \sigma_1 \sigma_2, \quad \SS = \SS_1 \cap \SS_2.
\]

For prime $q$ and $n \in \Z$, define the random set
\be\label{APH}
\mathbf{AP}(J;q,n) \coloneqq \{ n+qh: 1 \leq h \leq J\} \cap \mathbf{S}_1
\ee
that describes a portion of the progression $n \pmod{q}$ that survives the sieving process up to $H^{M}$.
Let $K\ge 2$ be a fixed integer parameter, which we will eventually take to be very large.
Given $\SS_1$, the probability that 
$\mathbf{AP}(KH;q,n) \subset \SS_2$ is about
$\sigma_2^{|\mathbf{AP}(KH;q,n)|}$, and if this occurs then removing the residue class $n\mod{q}$ will remove an essentially maximal number of elements.
  Central to our argument is the weight function
\begin{equation}\label{lambda-def}
 \bm{\lambda}(H;q,n) \coloneqq \begin{cases}
 \displaystyle \frac{1}{\sigma_2^{|\mathbf{AP}(KH;q,n)|}} &
 \text{ if } \mathbf{AP}(KH; q,n ) \subset \mathbf{S}_2, \\ 0 & \text{ otherwise.} \end{cases}
\end{equation}
Informally, $\bm{\lambda}(H;q,n)$ then isolates those $n$ with the (somewhat unlikely) property that the portion $\mathbf{AP}(KH;q,n)$ of the arithmetic progression
$n\mod q$ that survives the sieving process up to $H^{M}$, in fact also survives the sieving process all the way up to $z$.  
The weight nearly exactly counteracts the probability of this event, so that we anticipate $\bm{\lambda}(H;q,n)$ to
 be about 1 on average over $n$.
  In addition, $\bm{\lambda}(H;q,n)$ is skewed to be large for those $n$ with $\mathbf{AP}(KH;q,n)$ large.
  We will focus attention on those $n$ satisfying
  \[
  -Ky <n \le y,
  \]
  for outside this interval, if $q\in \mathcal{Q}_H$ then
  $\mathbf{AP}(KH;q,n)$ does not intersect the interval
  $[1,y]$ of primary interest.

Our aim is thus first select a random $\mathbf{b}\in\Z/P\Z$, and show that with high probability the random sets $\mathbf{S}_1$ and $\mathbf{S}_2$ behave as we expect for all scales $H\in\mathfrak{H}$. This implies that there is a good fixed choice $b\in\Z/P\Z$ where the (now deterministic) function $\lambda(H_q;q,n)$ is suitably concentrated on residue classes $n\mod{q}$ which contain many elements in $S=S_z+b$, for all $q$ in a suitable subset $\mathcal{Q}'\subseteq\mathcal{Q}$. In particular, this means that if we then select a residue class $n_q\mod{q}$ randomly with probability proportional to $\lambda(H_q;q,n)$, this residue class will typically contain many elements of $S$, for any $q\in\mathcal{Q}'$.

This is now precisely the situation of our hypergraph covering lemma, which we can then apply essentially as a black box. (The lemma is a minor variation of the one used in \cite{FGKMT} based on the ``R\"odl nibble'' or ``semi-random'' method; the proof is given in the appendix.) The conclusion from the lemma allows us to deduce that there is a choice of residue classes $n_q\mod{q}$ for $q\in\mathcal{Q}'$ which cover almost all of $S$. If we then choose $b\mod P(z,x/2)$ such that $b=n_q\mod{q}$ for all $q\in\mathcal{Q}'$ we then obtain \eqref{first_reduction}, and hence the result.

The paper is organized as follows. Theorem \ref{main} has previously been reduced to that of establishing \eqref{first_reduction}.  We will then reduce this task further to that of establishing Theorem \ref{second} (Second reduction) in the next section. In turn, Theorem \ref{second} will be reduced to
Theorem \ref{third} (Third reduction) in the following section. The final section is then dedicated to establishing Theorem \ref{third}.

\section{Greedy sieving via Hypergraph covering}
In this section we use our hypergraph covering lemma (Lemma \ref{hcl}, given below) to reduce the proof of Theorem \ref{main} to the claim that there is a good choice of $b$ for the initial sieving, which is given by Theorem \ref{second} below.

Recall the definition \eqref{Hq} of $H_q$ and that $S$ is the set $S_z+b$ depending on $b$.
\begin{theorem}[Second reduction]\label{second}
\red{Fix $\delta$ satisfying \eqref{delta1}, let $M-6$, $\xi-1$ be sufficiently small (in terms of $\delta$), $K$ sufficiently large in terms of $\delta$, and
$0 < \eps < \frac16 (M-6)$.  If $x$ is large enough, in terms of $\delta,M,\xi,K,\eps$,} then
there exists an integer $b$
and a set ${\mathcal Q}' \subset \mathcal{Q}$
such that
\begin{itemize}
\item[(i)]  one has
\begin{equation}\label{bsy-equi}
 |S \cap [1, y]| \le 2 \sigma y,
\end{equation}
\item[(ii)]  for all $q \in {\mathcal Q}'$, one has
\begin{equation}\label{laq}
 \sum_{-Ky<n\le y} \lambda(H_q;q,n) = \(1 + O\pfrac{1}{(\log x)^{\delta(1+\eps)} }\) (K+1)y,
\end{equation}
\item[(iii)]  for all but at most $\frac{ \rho x}{8 \log x}$ elements $n$ of $S \cap [1, y]$, one has
\begin{equation}\label{qhv}
 \sum_{q \in {\mathcal Q}'} \sum_{h \leq K H_q} \lambda(H_q; q, n - qh ) =  \(C_2 + O\pfrac{1}{(\log x)^{\delta(1+\eps)} }\) (K+1)y
\end{equation}
for some quantity $C_2$ independent of $n$ with
\begin{equation}\label{c1b}
 10^{2\delta} \leq C_2 \le 100.
\end{equation}
\end{itemize}
\end{theorem}
Theorem \ref{second} is saying that there is a good choice of $b\in \Z/P\Z$ such that we can then perform the second sieving stage effectively. The conclusions are what we would expect for ``typical'' $b$, so this merely sets the stage for the greedy sieve.

If we remove a residue class $\mathbf{n_q}\mod{q}$ where $\mathbf{n_q}$ is chosen randomly proportional to $\lambda(H_q;q,\cdot)$, then together \eqref{laq} and \eqref{qhv} say that the expected number of times $n\in S\cap[1,y]$ is removed is about $C_2>1$ (apart from a small exceptional set of $n$). This means that if we could realize these random variables so that the behavior was very close to this expectation, we would sieve in a perfectly uniform manner and would successfully remove almost all of $S\cap[1,y]$. The fact that we can pass from the random variables to such a uniform sieve is a consequence of the hypergraph covering lemma. It is vital that $C_2>1$, and the fact that we will ultimately succeed with $C_2$ bounded (rather than of size $\log\log{x}$) corresponds to us being able to take $y$ as large as $x(\log{x})^\delta$.

The fact that we have good error terms in the asymptotics and the slightly stronger lower bound $C_2>10^{2\delta}$ is needed for our hypergraph covering lemma, but this is not a limiting feature of our argument.

Another way to look at Theorem \ref{second} is that equation \eqref{laq} says that $\lambda(H_q;q,n)$ is about $1$ on average.  However,
when $n$ is drawn from the smaller set $S \cap [1, y]$ (which has density $\approx \sigma$ in
$[1,y]$), the quantity $\lambda(H_q;q, n-qh)$ appearing in
\eqref{qhv} is biased to be a bit larger (in our construction, it will eventually behave like
$\frac{\log y}{\log(y/q)}$ on the average over $q\in{\mathcal Q}'$), since $n\in AP(KH;q,n-hq)$ is already known to lie in $S$.  It is this bias that
ultimately allows us to gain somewhat over the trivial bound of $\gg x$ on the gap size in
Theorem \ref{main}.

To reduce Theorem \ref{main} to Theorem \ref{second}, we will use the following hypergraph covering lemma.

\begin{lemma}[Hypergraph covering lemma]\label{hcl} 
Suppose that $0 < \delta \le \frac12$, \red{$K>1$,
let $y\ge y_0(\delta,K)$ with $y_0(\delta,K)$} sufficiently large, and let $V$ be finite set with $|V|\le y$. 
Let $1\le s\le y$, and suppose that
$\mathbf{e}_1,\ldots,\mathbf{e}_s$ are random subsets of
$V$ satisfying the following:
\begin{align}
|\mathbf{e}_i| &\leq \red{\frac{K(\log y)^{1/2}}{\log\log y}} \qquad (1\le i\le s),\label{size-bound} \\
\PR( v \in \mathbf{e}_i ) &\leq y^{-1/2 - 1/100} \qquad
(v\in V, 1\le i\le s), \label{sparsity}\\
\sum_{i=1}^s \PR( v, v' \in \mathbf{e}_i ) &\leq y^{-1/2} \qquad (v,v'\in V, v\ne v'),\label{codegree}\\
\Bigg|\sum_{i=1}^s \PR( v \in \mathbf{e}_i ) -  C_2\Bigg| &\le  \eta \qquad (v\in V),\label{uniform} 
\end{align}

where $C_2$ and $\eta$ satisfy 
\begin{equation}\label{c2-bound}
 10^{2\delta} \leq C_2 \le 100, \qquad \eta \ge \frac{1}{(\log y)^{\delta}\log\log y}.
\end{equation}
Then there are subsets $e_i$ of $V$, $1\le i\le s$,
with $e_i$ being in the support of $\mathbf{e}_i$
for every $i$, and such that
\begin{equation}\label{vbound}
 \Big| V \backslash \bigcup_{i=1}^s e_i \Big| \le C_3 \eta |V|,
\end{equation}
where $C_3$ is an absolute constant.
\end{lemma}

This lemma is proven using almost exactly the same argument used to prove \cite[Corollary 4]{FGKMT} (after some minor changes of notation); we defer the proof to the appendix.

The conditions \eqref{size-bound}, \eqref{sparsity} and \eqref{codegree} should be thought of as conditions which ensure that the randoms sets $\mathbf{e}_i$ typically spread out and cover most vertices in $V$ fairly evenly. The condition \eqref{c2-bound} ensures that typically all vertices are covered slightly more than once in a uniform manner. Provided these conditions are fulfilled then the conclusion \eqref{vbound} is that there is a non-zero probability that virtually all vertices are covered, and so there is a deterministic realization of the random variables which covers virtually all the vertices. The key point is that $C_2$ can be taken to be bounded, since this means that the covering sets $e_i$ are close to disjoint, and this is what allows us to improve the situation of trying to sieve independently for each $q$.

\begin{proof}[Reduction of Theorem \ref{main} to Theorem \ref{second}]
We are now in a position to deduce \eqref{first_reduction}, and hence Theorem \ref{main}, from
Theorem \ref{second}. Let $b$ and $\mathcal{Q}'$ be the quantities whose existence is asserted by Theorem \ref{second}, and so $S=S_z+b$.

Property (iii) of Theorem \ref{second} implies that there is a set
$V \subseteq S \cap [1, y]$ ,
containing all but at most $\frac{\rho x}{8 \log x}$ elements
of $S \cap [1, y]$, and such that  \eqref{qhv} holds
for all $n\in V$.
   For each $q \in {\mathcal Q}'$, we choose a random integer $\mathbf{n}_q$ with probability
density function
\begin{equation}\label{ndens}
 \PR( \mathbf{n}_q = n ) = \frac{\lambda(H_q;q,n)}{\sum_{-Ky<n'\le y} \lambda(H_q;q,n')}.
\end{equation}
Note that by \eqref{laq} that the denominator is non-zero, so that this is a well-defined probability distribution.
   We will not need to assume any independence hypotheses on the $\mathbf{n}_q$.  For each $q \in {\mathcal Q}'$, we then define the random subset $\mathbf{e}_q$ of $V$ by the formula
\begin{equation}\label{eqdef}
\mathbf{e}_q \coloneqq V \cap \{ \mathbf{n}_q + hq: 1 \leq h \leq KH_q \}.
\end{equation}
Our goal is to show that there are choices $n_q$ of the random variable $\mathbf{n}_q$ which occur with positive probability such that the corresponding sets $e_q$ cover most of $V$. Specifically, we wish to use Lemma \ref{hcl} to show that
\be\label{Veq}
\bigg| V \backslash \bigcup_{q \in {\mathcal Q}'} e_q \bigg| \leq \frac{\rho x}{8 \log x}.
\ee
By construction, if \eqref{Veq} holds then for each $q\in \mathcal{Q}'$ there is a number $n_q$ such that
$$ e_q \subset \{ n \in V: n \equiv n_q \pmod{q} \}.$$
Taking $b\equiv n_q \pmod{q}$ for all $q \in {\mathcal Q}'$, we find that
\[
\big| (S_{x/2}+b) \cap [1,y] \big| \le |S\cap[1,y]\backslash V|+\bigl| V \backslash \bigcup_{q \in {\mathcal Q}'} e_q \Bigr|\le \frac{\rho x}{8 \log x} + \frac{ \rho x}{8 \log x} = \frac{\rho x}{4 \log x},
\]
as required for \eqref{first_reduction}.
The fractions $\frac18$ and $\frac14$ above are irrelevant to
the determination of the best exponent in Theorem \ref{main},
and were chosen for convenience.

Thus it remains to construct $e_q$ satisfying \eqref{Veq}, and this is
accomplished by Lemma \ref{hcl}.  We wish to apply Lemma \ref{hcl} with $s=|\mathcal{Q}'|$,
$\{\mathbf{e}_1,\ldots,\mathbf{e}_s\} = \{ \mathbf{e}_q : q\in \mathcal{Q}'\}$, $C_2$ as given by Theorem \ref{second}, and
\[
\eta = \frac{\rho/20}{C_3 (\log x)^{\delta}}.
\]
With this choice of parameters we see from \eqref{bsy-equi}, \eqref{Mertens}, and \eqref{y-def} that
$$ C_3 \eta |V| \le  \frac{\rho/10}{(\log x)^{\delta}} \frac{y}{\log z} \sim (\rho/10)   \frac{x}{\log x}.$$
Hence, \eqref{Veq} follows from \eqref{vbound} if $x$ is large enough.   Thus, it suffices to verify the hypotheses \eqref{size-bound}, \eqref{sparsity}, \eqref{codegree}, \eqref{uniform} and \eqref{c2-bound} of the lemma, which we accomplish using the conclusions \eqref{laq} and \eqref{qhv} of
Theorem \ref{second}.

Note that if $q \in {\mathcal Q}'$, then from \eqref{eqdef} and \eqref{H-def} we have
\[ 
\red{|\mathbf{e}_q| \leq KH_q \leq \frac{Ky}{z} = \frac{K(\log x)^{1/2}}{\log\log x} \le \frac{K(\log y)^{1/2}}{\log\log y}}
\]
which gives \eqref{size-bound}.  Similarly, for $n \in V$ and $q \in {\mathcal Q}'$, we have from \eqref{eqdef}, \eqref{ndens}, and \eqref{lambda-def} that
\begin{align*}
\PR( n \in \mathbf{e}_q ) &= \sum_{1 \leq h \leq KH_q} \PR( \mathbf{n}_q = n - h q ) \\
&\ll \frac{1}{y} \sum_{1 \leq h \leq KH_q} \lambda(H_q;q, n-hq) \\
&\ll \frac{1}{y} H_q \sigma_2^{-H_q}
\ll \frac{1}{y^{9/10}}
\end{align*}
which gives \eqref{sparsity} for $y$ large enough.

Applying \eqref{eqdef}, \eqref{ndens}, \eqref{laq}, and \eqref{qhv} successively
yields 
\begin{align*}
 \sum_{q \in { \mathcal{Q}'} }\PR( v \in \mathbf{e}_q ) &= \sum_{q\in\mathcal{Q}'}
 \sum_{h\le KH_q} \PR(\mathbf{n}_q = v-hq) \\
&=  \sum_{q\in\mathcal{Q}'}
 \sum_{h\le KH_q} \frac{\lambda(H_q;q,v-hq)}{\sum_n \lambda(H_q;q,n)} \\
  &= C_2 + \red{O( (\log x)^{-(1+\eps)\delta})},
\end{align*}
and \eqref{uniform} follows.
We now turn to \eqref{codegree}.  Observe from \eqref{eqdef} that for distinct $v,v' \in V$, one can only have $v,v' \in \mathbf{e}_q$ if $q$ divides $v-v'$.  Since $|v-v'| \le 2y$ and $q \geq z > \sqrt{2y}$, there is at most one $q$ for which this is the case, and \eqref{codegree} now follows from \eqref{sparsity}.
This concludes the derivation of \eqref{first_reduction} from Theorem \ref{second}.
\end{proof}
To complete the proof of Theorem \ref{main}, we need to
prove  Theorem \ref{second} and Lemma \ref{hcl}.
The proof of Theorem \ref{second} depends on various first and second moment
estimations of the weights, which are given in the next two sections.  The proof of Lemma \ref{hcl} will occupy the Appendix.

 %
 %

\section{Concentration of $\lambda(H;q,n)$}
\label{sec:concentration}

 %
 %
 %

In this section, we
 deduce Theorem \ref{second} from the following moment calculations.

\begin{theorem}[Third reduction]\label{third}
Assume that $M\ge 2$.  Then
\begin{itemize}
\item[(i)]  One has
\begin{align}
\E  |\SS \cap [1, y]| &= \sigma y,  \label{ifirst} \\
  \E |\SS \cap [1, y]|^2 &= \left(1+O\left(\frac{1}{\log y}\right)\right) (\sigma y)^2. \label{szyj}
\end{align}
\item[(ii)]  For every $H\in \mathfrak H$, and for $j\in \{0,1,2\}$ we have
\be
 \E \sum_{q \in {\mathcal Q}_H} \left(\sum_{-Ky<n\le y} \bm{\lambda} (H;q,n)\right)^j = \left(1+O\left(\frac{1}{H^{M-2}}\right)\right) ((K+1)y)^j |{\mathcal Q}_H|.
 \label{ii}
\ee
\item[(iii)]
For every $H\in \mathfrak H$, and for $j\in \{0,1,2\}$ we have
\begin{equation}\label{da}
 \E \sum_{n \in \SS \cap [1, y]} \Bigg(\sum_{q \in {\mathcal Q}_H} \sum_{h \leq KH} \bm{\lambda}(H; q, n-qh )\Bigg)^j = \left(1+O\left(\frac{1}{H^{M-2}}\right)\right) \left(\frac{|{\mathcal Q}_H| \red{\cdot  \fl{KH} }}{\sigma_2 }\right)^j \sigma y.
\end{equation}
\end{itemize}
\end{theorem}
We remind the reader that in Theorem \ref{third} the random variables $\mathbf{S}$ and $\bm{\lambda}$ are defined in terms of the random variable $\mathbf{b}$ chosen uniformly in $\Z/P\Z$, not the random variables $\mathbf{n_q}$ we encountered in the previous section.

Note that for every $n \in [1,y]$ and $h \leq KH$ we have $n-qh \in \red{(-Ky,y]}$, so the quantity in \eqref{da} is well-defined.  As with the previous theorem, the quantity $\boldsymbol{\lambda}(H;q,n$) behaves like $1$ on the average when $n$ is drawn from $[-Ky,y] \cap \Z$, but for $n$ drawn from $\SS \cap [1, y]$ (in particular, $n\in \SS_2$), the quantity $\boldsymbol{\lambda}(H; q, n-qh )$ is now biased to have an average value of approximately $\sigma_2^{-1}$ because $n-qh+qh=n$ is automatically
in $\SS_2$; recall the definition \eqref{lambda-def} of $\lambda(H; q, n-qh )$.

\begin{proof}[Deduction of Theorem \ref{second} from Theorem \ref{third}]
We draw $\bb$ uniformly at random from $\Z/P\Z$.  It will suffice to generate a random set $\bm{\mathcal Q}'$ such that the random function $\bm{\lambda}$ defined in \eqref{lambda-def} satisfies the conclusions of Theorem \ref{second} (with $b$ replaced by $\bb$) hold with positive probability - in fact, we will show that they hold with probability $1+o(1)$.

Assume that $M$ satisfies \eqref{M}.
  From Theorem \ref{third}(i) we have
$$  \E \big||\SS \cap [1, y]| - \sigma y\big|^2 \ll \frac{(\sigma y)^2}{\log y}.$$
Hence by Chebyshev's inequality, we see that
\be\label{Snorm}
\PR \(|\SS \cap [1, y]| \le 2\sigma y \)
=1 - O(1/\log x),
\ee
verifying \eqref{laq} in Theorem \ref{second}.
Let $H \in {\mathfrak H}$.
From Theorem \ref{third}(ii) we have  (recall that our implied constants may depend on $K$)
\be\label{sumlamqn}
 \E \sum_{q \in {\mathcal Q}_H} \Bigg(\sum_{-Ky<n\le y} \bm{\lambda}(H;q,n) - (K+1)y\Bigg)^2 \ll \frac{y^2 |{\mathcal Q}_H|}{H^{M-2}}.
 \ee
Now let $\bm{\mathcal Q}'_H$ be the subset of $q\in {\mathcal Q}_H$
with the property that
\begin{equation}\label{manx}
\Big| \sum_{-Ky < n \le y} \bm{\lambda}(H;q,n)- (K+1)y \Big| \le \frac{y}{H^{1+\eps}}.
\end{equation}
It follows from \eqref{sumlamqn} and \eqref{manx} that
\be\label{QHQH'}
\E  |{\mathcal Q}_H \backslash {\bm{\mathcal Q}}'_{H}| \ll \frac{|\mathcal{Q}_H|}{H^{M-4-2\eps}}.
\ee
By Markov's inequality, it follows that
with probability $1 - O(H^{-\eps})$, one has
\begin{equation}\label{q-excep}
 |{\mathcal Q}_H \backslash {\bm{\mathcal Q}}'_{H}| \ll \frac{ |{\mathcal Q}_H| }{H^{M-4-3\eps}}.
\end{equation}
By \eqref{M}, we have $M>4+3\eps$ for small enough $\eps$, that is, the exponent in the denominator in
\eqref{q-excep} is positive.
Since $\sum_{H\in \mathfrak{H}} H^{-\eps} \ll (y/x)^{-\eps} \ll (\log x)^{-\delta \eps}$, with probability $1-O((\log x)^{-\delta \eps})$
the relation \eqref{q-excep} holds for every $H\in \mathfrak{H}$
simultaneously.  We now set
$$ {\bm{\mathcal Q}'} \coloneqq \bigcup_{H \in {\mathfrak H}} {\bm {\mathcal Q}}'_{H}.$$
Then, on the probability $1-o(1)$ event that \eqref{q-excep} holds for every $H$
and that \eqref{Snorm} holds, items (i) \eqref{bsy-equi} and (ii) \eqref{laq}
of Theorem \ref{second} follow upon recalling \eqref{manx} and the lower
bound $H\gg (\log x)^{\delta}$.

We work on part (iii) of Theorem \ref{second} using Theorem \ref{third}(iii) in a similar fashion to previous arguments.  We have
$$ \E \sum_{n \in \mathbf{S} \cap [1, y]} \left|\sum_{q \in {\mathcal Q}_H} \sum_{h \leq KH} \bm{\lambda}(H; q, n-qh ) -  \frac{|{\mathcal Q}_H| \red{\cdot \fl{K H}}}{\sigma_2}\right|^2 \ll \frac{1}{H^{M-2}} \left( \frac{|{\mathcal Q}_H| \red{\cdot \fl{K H}}}{\sigma_2}\right)^2 \sigma y.$$
If we let $\mathcal{E}_H$ denote the set of $n\in \mathbf{S} \cap [1, y]$ such that
\be\label{EH}
\left|\sum_{q \in {\mathcal Q}_H} \sum_{h \leq KH} \bm{\lambda}(H; q, n-qh ) -  \frac{|{\mathcal Q}_H|\red{\cdot \fl{KH}}}{\sigma_2}\right| \ge \red{\frac{|{\mathcal Q}_H| \cdot \fl{KH}}{\sigma_2 H^{1+\eps}}},
\ee
\red{then, recalling that $M>6$ and $\eps$ is very small,}
\[
\red{\E |\mathcal{E}_H| \ll \frac{\sigma y}{H^{1+2\eps}}.}
\]
\red{ By Markov's inequality, we conclude that $|\mathcal{E}_H| \le \sigma y/H^{1+\eps}$ with probability $1-O(H^{-\eps})$.}

 We next estimate the contribution from ``bad'' primes $q\in
{\mathcal Q}_H \backslash {\bm {\mathcal Q}}'_H$.
For any $h\le H$, by Cauchy-Schwarz we have 
\[
\E \sum_{n \in \mathbf{S} \cap [1, y]}  \;\; \sum_{q \in {\mathcal Q}_H \backslash {\bm {\mathcal Q}}'_H} \bm{\lambda}(H;q,n-hq) \le\( \E |{\mathcal Q}_H \backslash {\bm {\mathcal Q}}'_H| \)^{1/2} \( \E  \!\! \sum_{q \in {\mathcal Q}_H \backslash {\bm {\mathcal Q}}'_H} \Big| \red{\sum_{-Ky < n-hq \le y}} \!\! \bm{\lambda}(H;q,n-hq) \Big|^2 \)^{1/2}
\]
and by the triangle inequality, \eqref{sumlamqn} and \eqref{QHQH'},
\begin{align*}
\E  \sum_{q \in {\mathcal Q}_H \backslash {\bm {\mathcal Q}}'_H} \bigg| & \red{\sum_{-Ky < n-hq \le y}} \!\!\bm{\lambda}(H;q,n-hq) \bigg|^2  \\
&\qquad \le 2  \E \sum_{q \in {\mathcal Q}_H \backslash {\bm {\mathcal Q}}'_H}  \Bigg( \Big| \red{\sum_{-Ky < n-hq \le y}} \!\!\!\!\!\!\bm{\lambda}(H;q,n-hq)-(K+1)y \Big|^2 + (K+1)^2y^2 \Bigg)
\\ & \qquad \ll \frac{y^2|\mathcal{Q}_H|}{H^{M-4-2\eps}}.
\end{align*}
Therefore, by  \eqref{QHQH'} and summing over $h\le KH$,
\[
\E  \sum_{n \in \mathbf{S} \cap [1, y]}   \sum_{q \in {\mathcal Q}_H \backslash {\bm {\mathcal Q}}'_H} \sum_{h\le KH} \bm{\lambda}(H;q,n-hq) \ll \frac{y|\mathcal{Q}_H|}{H^{M-5-2\eps}}.
\]
Let $\mathcal{E}_H'$ denote the set of $n \in \mathbf{S} \cap [1, y]$
so that
\be\label{EH'}
\sum_{q \in {\mathcal Q}_H \backslash {\bm {\mathcal Q}}'_H} \sum_{h\le KH} \bm{\lambda}(H;q,n-hq) \ge \red{\frac{|\mathcal{Q}_H| \cdot \fl{KH}}{H^{1+\eps} \sigma_2}.}
\ee
Then
\[
\E |\mathcal{E}_H'| \ll \red{\frac{y H^{1+\eps} \sigma_2}{H^{M-4-2\eps}} \ll \sigma y \frac{\log H}{H^{M-5-3\eps}}.}
\]
By Markov's inequality, \red{$|\mathcal{E}_H'| \le \sigma y/H^{1+\eps}$ with
probability $1-O(1/H^{M-6-5\eps})$.}
By \eqref{M} again, if
$\eps$ is small enough then \red{$M-6-5\eps>\eps$}.
Consider the event that \eqref{Snorm} holds,
and that for every $H$, we have
\eqref{q-excep},  \red{$|\mathcal{E}_H| \le \sigma y/H^{1+\eps}$} and
 $|\mathcal{E}_H'| \le \sigma y/H^{1+\eps}$.
This simultaneous event happens with positive probability on account of
$\sum_{H\in \mathfrak{H}} H^{-\eta} \ll (\log x)^{-\delta \eta}$ for any $\eta>0$.  As mentioned before, items (i) and (ii) of Theorem \ref{second}
hold.  Now let
\[
\mathcal{N} = \mathbf{S} \cap [1, y] \backslash \bigcup_{H\in \mathfrak{H}} (\mathcal{E}_H \cup \mathcal{E}_H').
\]
The number of exceptional elements satisfies
$$
\Bigg| \bigcup_{H\in\mathfrak{H}} (\mathcal{E}_H \cup \mathcal{E}_H') \Bigg|\ll \frac{\sigma y}{(\log x)^{\delta(1+\eps)}},
$$
 which is smaller than $\frac{ \rho x}{8\log x}$
for large $x$.  It remains to verify \eqref{qhv} for $n\in \mathcal{N}$.
Since $n\not\in \mathcal{E}_H$ and $n\not\in \mathcal{E}_H'$ for every $H$,
the inequalities opposite to those in \eqref{EH} and \eqref{EH'} hold, and we have for each $H\in \mathfrak{H}$
the asymptotic 
\[
\sum_{q \in {\mathcal Q}'_H} \sum_{h \leq KH} \bm{\lambda}(H; q, n-qh ) = \(1+\red{O\pfrac{1}{H^{1+\eps}}}\) \frac{|{\mathcal Q}_H| \red{\cdot \fl{KH}}}{\sigma_2}.
\]
Therefore,
\begin{align*}
\sum_{q \in \bm{\mathcal Q}'} \sum_{h \leq KH_q} \bm{\lambda}( H_q;q, n - qh ) &=
\sum_{H \in {\mathfrak H}} \sum_{q \in \bm{\mathcal Q}'_H} \sum_{h \leq KH} \bm{\lambda}(H; q, n - qh ) \\
&=
\left(1 + O\left(\frac{1}{(\log x)^{(1+\eps)\delta}}\right)\right) C_2 \times (K+1) y
\end{align*}
where
$$ 
\red{C_2 \coloneqq \frac{1}{(K+1)y} \sum_{H \in {\mathfrak H}} \frac{|{\mathcal Q}_H| \cdot \fl{KH}}{\sigma_2}.}
$$
This verifies \eqref{qhv}.
From \eqref{QH}, we see that $C_2$ does not depend on $n$
($C_2$ depends only on $x$).
Using \eqref{Mertens} and \eqref{QH},  
\[
C_2 \sim \frac{K}{(K+1)y} \rho(1-1/\xi) \sum_{H \in {\mathfrak H}} \frac{\log z}{\log(H^{M})} \; \frac{yH}{H\log x} \qquad (x\to \infty).
\]
Recalling the definitions \eqref{y-def} of $y$ and
\eqref{z-def} of $z$, together with the bounds \eqref{H-bounds} on $H$, we thus have
as $x\to \infty$,
\begin{align*}
C_2 &\sim \frac{K\rho(1-1/\xi)}{M(K+1)} \sum_{H \in {\mathfrak H}} \frac{1}{\log H} \\
&= \frac{K\rho(1-1/\xi)}{M(K+1)} \sum_{2 (\log x)^{\delta} \leq \xi^j \leq \xi^{-1}(\log x)^{1/2}/\log\log x} \frac{1}{j \log \xi}.
\end{align*}
Summing on $j$ we conclude that
\[
C_2 \sim \frac{K\rho}{M(K+1)} \, \frac{1-1/\xi}{\log \xi}
\log\pfrac{1}{2\delta}. 
\]
Finally, recalling \eqref{delta1}, we see that if $K$ is large enough, $\xi$ is sufficiently close to 1 and \red{$M$ sufficiently close to $6$}, then
\[
C_2 \ge 10^{2\delta},
\]
as required for \eqref{c1b}.
\end{proof}

\begin{remark}\label{rem:limit}
The limit our methods appears to be an exponent $e^{-1/\rho}-o(1)$ in Theorem \ref{main}.  Such a bound assumes that we
may succeed with the previous argument for any choice of $M>1$, any $C_2>1$ and with $z=y/(\log x)^{1+o(1)}$ in place of
of $z=y/(\log x)^{1/2+o(1)}$.   Then the above calculation reveals that $C_2>1$ provided $\rho\log(1/\delta)>1$.
 Each of these
 conditions appears to be essential in the succeeding arguments in the next sections.
\end{remark}

\medskip

It remains to establish Theorem \ref{third}.  This is the objective of the next section of the paper.

%
%
\section{Computing correlations}
%
%

In this section, we verify the claims in Theorem \ref{third}. We will frequently need to compute $k$-point correlations of the form
$$ \PR\left( n_1,\dots,n_k \in \mathbf{S}_2 \right )$$
for various integers $n_1,\dots,n_k$ (not necessarily distinct).
Heuristically, since $\mathbf{S}_2$ 
avoids $I_p$ residue classes modulo $p$ for each $p$,
we expect that the above probability is roughly
$\sigma_2^k$ for typical choices of $n_1,\dots,n_k$. 
Unfortunately, there is some fluctuation from this prediction, most obviously when two or more of the $n_1,\dots,n_k$ are equal, but also if the reductions $n_i\pmod{p}, n_j\pmod{p}$ for some prime $p \in (H^{M},z]$ have the same difference as
two elements of $I_p$.  Fortunately we can control these fluctuations to be small on average.  To formalize this statement we need some notation.   Let ${\mathcal D}_H \subset \N$ denote the collection of squarefree numbers $d$, all of whose prime factors lie in $(H^{M},z]$.  This set includes $1$, but we will frequently remove $1$ and work instead with ${\mathcal D}_H \backslash \{1\}$.  For each $d \in {\mathcal D}_H$, let $I_d \subset \Z/d\Z$ denote the collection of residue classes $a\mod d$ such that $a\mod p \in I_p$ for all $p\mid d$.  Recall the defnition of the difference set $\mathcal{A}-\mathcal{B} := \{ a-b : a\in \mathcal{A},b\in \mathcal{B} \}$.  For any integer $m$ and any parameter $A>0$, we define the error function
\begin{equation}\label{eam}
 E_A(m;H) \coloneqq \sum_{d \in {\mathcal D}_H \backslash \{1\}} \frac{A^{\omega(d)}}{d} 1_{m\kern-2pt\pmod d \in I_d - I_d},
\end{equation}
where $\omega(d)$ is the number of prime factors of $d$.
The quantity $E_A(m;H)$ looks complicated, but in practice it will be quite small on average over $m$.  We also observe that $E_A$ is an even function: $E_A(-m;H) = E_A(m;H)$.

Before we start our proof of Theorem \ref{third}, we first need two preparatory lemmas. The following lemmas hold for general $H$, not necessarily restricted to $H\in \mathfrak{H}$.  Recall that implied constants in $O-$ may depend on $B$ and $M$.

\begin{lemma}\label{correlation} 
Let $10<H<z^{1/M}$, $1 \leq \ell \le 10KH$, and suppose
that $\mathcal{U} \subset \mathcal{V}$
are finite sets of integers with $|\mathcal{V}|=\ell$.
 Then we have
$$ \PR( \mathcal{U} \subset \mathbf{S}_2  ) = \sigma_2^{|\mathcal{U}|} \Bigg(1 + O\pfrac{|\mathcal{U}|^2}{H^{M}}+ O\Bigg(\frac{1}{\ell^2} \ssum{v,v'\in \mathcal{V}\\ v\ne v'} E_{2\ell^2B}(v-v';H)\Bigg)\Bigg).$$
\end{lemma}

\begin{remark}
The numbers in $\mathcal{V} \setminus \mathcal{U}$ 
are ``dummy variables'',
 but it is often convenient to include them.
Typically, $\mathcal{U}$ will be an irregular subset, with  unknown size, of a
regular set $\mathcal{V}$, whose size is known.
We often have better control of the error averaged over the larger set.
\end{remark}

\begin{proof}  For each prime $p \in (H^{M},z]$, let $\mathbf{b}_{2,p} \in \Z/p\Z$ be the reduction of $\mathbf{b}_2$ modulo $p$, thus each $\mathbf{b}_{2,p}$ is uniformly distributed in $\Z/p\Z$ and the $\mathbf{b}_{2,p}$ are independent in $p$.
Let $N_p$ denote the set of residue classes $\mathcal{U} \pmod{p}$.
 By the Chinese Remainder Theorem, we thus have
\begin{align*}
\PR( \mathcal{U} \subset \mathbf{S}_2 ) &= \prod_{p \in (H^{M},z]} \PR( N_p \cap ( \mathbf{b}_{2,p} + I_p)=\emptyset ) \\
&=
\prod_{p \in (H^{M},z]} \left(1 - \PR( \mathbf{b}_{2,p} \in N_p - I_p)\right) \\
&=
\prod_{p \in (H^{M},z]} \left(1 - \frac{|N_p - I_p|}{p}\right) .
\end{align*}
Let $k=|\mathcal{U}|$.
We may crudely estimate the size of the difference set $N_p - I_p$ by
$$
k|I_p| \ge |N_p - I_p| \ge k |I_p| - |I_p| \sum_{u,u'\in \mathcal{U}, u\ne u'} 1_{u-u' \kern-2pt\pmod{p}\in I_p - I_p}.
$$
Since $|I_p| \leq B$ and \red{$k \le 10 K H$}, 
we have $k|I_p| < 10KBH < p/10$ for $x$ large enough 
in terms of $M$.
Thus,
\begin{align*}
\left(1 - \frac{|N_p - I_p|}{p}\right)
=\(1-\frac{k|I_p|}{p}\) \(1+ \frac{k|I_p|-|N_p-I_p|}{p-k|I_p|}\)
=\(1-\frac{k|I_p|}{p}\) \Delta_p,
\end{align*}
where
\begin{align*}
1\le \Delta_p &\le 1 + \frac{2B}{p} \sum_{u,u'\in \mathcal{U}, u\ne u'} 1_{u-u'\kern-2pt \pmod{p}\in I_p - I_p} \\
&\le \prod_{u,u'\in \mathcal{U}, u\ne u'}\exp\left\{2B\frac{ 1_{u-u'\kern-2pt \pmod{p}\in I_p - I_p}}{p}\right\}\\&\le 
\prod_{v,v'\in\mathcal{V}, v\ne v'}\exp\left\{2B\frac{ 1_{v-v'\kern-2pt \pmod{p}\in I_p - I_p}}{p}\right\}.
\end{align*}
Here we have enlarged the summation over pairs of numbers 
from $\mathcal{V}$.
We have
\[
\prod_{H^M<p\le z} \(1-\frac{k|I_p|}{p}\) = \sigma_2^k 
\(1+O\pfrac{k^2}{H^M}\).
\]
By the arithmetic mean-geometric mean inequality, we have
\begin{align*}
\prod_{p \in (H^{M},z]} \Delta_p &\le  \prod_{v,v'\in \mathcal{V}, v\ne v'} \;\;
\prod_{p \in (H^{M},z]} \exp \Bigg\{2 B \frac{1_{v-v'\kern-2pt\pmod{p} \in I_p - I_p}}{p} \Bigg\} \\
&\le \frac{2}{\ell^2-\ell} \; \sum_{v,v'\in \mathcal{V}, v\ne v'} \;\;
 \prod_{p \in (H^{M},z]} \exp \Bigg\{2 B \pfrac{\ell^2-\ell}2\, \frac{1_{v-v'\kern-2pt\pmod{p} \in I_p - I_p}}{p} \Bigg\} \\
&\le\frac{2}{\ell^2-\ell} \; \sum_{v,v'\in \mathcal{V}, v\ne v'}\;\;
 \prod_{p \in (H^{M},z]} \(1+2B\ell^2
\frac{1_{v-v'\kern-2pt\pmod{p} \in I_p - I_p}}{p} \) .
\end{align*}
Recalling the definition \eqref{eam} of $E_A(n;H)$ 
we see that
\begin{align*}
\prod_{p \in (H^{M},z]} \Delta_p
&\le \frac{2}{\ell^2-\ell} \sum_{v,v'\in \mathcal{V}, v\ne v'} \big(1 + E_{2B\ell^2}(v-v';H) \big) \\
&= 1 + \frac{2}{\ell^2-\ell} \sum_{v,v'\in \mathcal{V}, v\ne v'}E_{2B\ell^2}(v-v';H).\qedhere
\end{align*}
\end{proof}

To estimate the average contribution of the errors $E_{2B\ell^2}(v-v')$ appearing in the above lemma, we will use the following estimate.

\begin{lemma}\label{ds} Suppose that $10<H<z^{1/M}$,
and that $(m_t)_{t \in T}$ is a sequence of integers indexed by a finite set $T$, obeying the bounds
\begin{equation}\label{sms}
 \sum_{t \in T} 1_{m_t \equiv a\pmod{d}} \ll  \frac{X}{\phi(d)} + R
\end{equation}
for some $X,R > 0$ and all $d \in {\mathcal D}_H \backslash \{1\}$ and $a \in \Z/d\Z$.  Then, for any $0<A$ satisfying $AB^2\le H^M$ and any integer $j$, one has
$$ \sum_{t \in T} E_{A}(m_t + j;H) \ll X \frac{A}{H^{M}} + R \exp\left( A B^2 \log \log y\right).$$
\end{lemma}

In practice, $R$ will be much smaller than $X$, and the first term on the right-hand side will
dominate.

\begin{proof}
From the Chinese Remainder Theorem and \eqref{ip-bound}, we see that for any $d \in {\mathcal D}_H$, we have
$$ |I_d| = \prod_{p\,\mid \, d} |I_p|\le B^{\omega(d)}.$$
In particular, the difference set $I_d - I_d \subset \Z/d\Z$ obeys the bound
$$
 |I_d - I_d| \leq B^{2\omega(d)}.
$$
From \eqref{eam}, \eqref{sms} we thus have
\begin{align*}
 \sum_{t \in T} E_{A}(m_t + j;H) &= \sum_{d \in {\mathcal D}_H \backslash \{1\}}
 \frac{A^{\omega(d)}}{d} \sum_{a\in I_d - I_d } \# \{t\in T : m_t+j \equiv a\pmod{d} \} \\ &\ll
\sum_{d \in {\mathcal D}_H \backslash \{1\}} \frac{(AB^2)^{\omega(d)}}{d} \left( \frac{X}{\phi(d)}
+ R \right).
\end{align*}
From Euler products and Mertens' theorem (for primes) we have
\[
\sum_{d \in {\mathcal D}_H} \frac{(AB^2)^{\omega(d)}}{d} = \prod_{p \in (H^{M},z]}
(1+AB^2/p) \le \exp \{AB^2\log\log y \}
\]
and
\[
\sum_{d \in {\mathcal D}_H} \frac{(AB^2)^{\omega(d)}}{d\phi(d)} = \prod_{p \in (H^{M},z]}
\(1+\frac{AB^2}{p^2-p}\) \le \exp \{AB^2/H^{M} \}\le 1+O(A/H^M).\qedhere
\]
\end{proof}

Finally, we are now in a position to complete the proof of Theorem \ref{third}.
%
%

\begin{proof}[Proof of Theorem \ref{third} (i)]
 By linearity of expectation, we have
$$  \E |\mathbf{S} \cap [1, y]| = \sum_{1\le n \leq y} \PR( n \in \SS ).$$
Since the set $S$ is periodic with period $P$ and has density $\sigma$, the summands here are all equal to $\sigma$, and \eqref{ifirst} follows.  Now we consider \eqref{szyj}. 
 Here we 
decompose $\SS$ as $\SS = \SS_1 \cap \SS_2$ using
\eqref{abbr-1} and \eqref{abbr-2} with 
$$H=\frac14 (\log y)^{1/M}.$$
 By the Prime Number Theorem,
\be\label{P1-parti}
P_1 = \exp\{(1+o(1)) H^M  \} \le   y^{1/4+o(1)}.
\ee
By linearity of expectation,
\begin{align*}
\E |\mathbf{S} \cap [1, y]|^2 &= \sum_{n_1,n_2 \leq y} \PR\left( n_1, n_2 \in \mathbf{S} \right) \\
&=\sum_{n_1,n_2 \leq y} \PR\left( n_1, n_2 \in \mathbf{S}_1 \right) \PR\left( n_1, n_2 \in \mathbf{S}_2 \right).
\end{align*}
 Observe that the probability $\PR\left( n_1,n_2 \in \mathbf{S}_1 \right)$ depends only on the reductions $\ell_1 :\equiv n_1\pmod{P_1}$, $\ell_2 :\equiv n_2\pmod{P_1}$.
Also, applying Lemma \ref{correlation} (with $\mathcal{U}=\mathcal{V}=\{n_1,n_2\}$), we have
\begin{align*}
\PR( n_1,n_2 \in \mathbf{S}_2 ) &= \red{\left( 1 + O(H^{-M})+ O\left(E_{8B}(n_1-n_2;H)\right)\right)} \sigma_2^2.
\end{align*}
Therefore,
\be\label{ES1y-a}
\begin{split}
\E |\mathbf{S} \cap [1, y]|^2 &= 
\sum_{1\le \ell_1,\ell_2 \le P_1} \PR\left( \ell_1, \ell_2 \in \SS_1 \right)
\ssum{1\le n_1,n_2\le y \\ n_1 \equiv \ell_1\pmod{P_1} \\ n_2 \equiv \ell_2 \pmod{P_1}} \PR(n_1,n_2\in \SS_2)\\
&= \sigma_2^2 \red{\(1+O\pfrac{1}{\log y}  \)} \sum_{1\le \ell_1,\ell_2 \le P_1} \PR\left( \ell_1, \ell_2 \in \mathbf{S}_1 \right) \(\frac{y}{P_1} + O(1)\)^2 +\\
&\qquad +O\Bigg(\sigma_2^2 \sum_{1\le \ell_1,\ell_2 \le P_1} \PR\left( \ell_1, \ell_2 \in \mathbf{S}_1 \right)
\ssum{1\le n_1,n_2\le y \\ n_1 \equiv \ell_1\pmod{P_1} \\ n_2 \equiv \ell_2 \pmod{P_1}} E_{8B}(n_1-n_2;H) \Bigg).
\end{split}
\ee
By the definition \eqref{abbr-1},
\be\label{sum-1}
  \sum_{1\le \ell_1,\ell_2 \le P_1} \PR\left( \ell_1,\ell_2 \in \mathbf{S}_1 \right) =
   \E \left| \mathbf{S}_1 \cap [1,P_1] \right|^2 = (\sigma_1 P_1)^2,
\ee
since $|\SS_1 \cap [1,P_1]|=\sigma_1 P$ always.
Next, fix $\ell_1, \ell_2 \in \Z/P_1\Z$.
Direct counting shows that for any $n_1$, natural number \red{$d\in {\mathcal D}_{H} \setminus \{1\}$} and residue class $a\mod d$, we have
$$
\# \{ n_2 \leq y : n_2 \equiv \ell_2 \pmod{P_1}, n_1-n_2 \equiv a\pmod{d} \}
 \ll  \frac{y}{dP_1} + 1 \le \frac{y}{\phi(d) P_1}+1.
$$
Applying Lemma \ref{ds} to the inner sum over $n_2$, we deduce that
\begin{equation}\label{sum-2}
\begin{split}
\ssum{1\le n_1,n_2\le y \\ n_1 \equiv \ell_1\pmod{P_1} \\ n_2 \equiv \ell_2 \pmod{P_1}}E_{8B}(n_1-n_2;H)
  &\ll \left(\frac{y}{P_1}\right)^2 \frac{1}{H^{M}} + \frac{y}{P_1} \exp\left( O( \log\log y) \right) \\
  &\ll \frac{y^2}{P_1^2 H^M} \red{\ll \frac{y^2}{P_1^2 \log y}}
  \end{split}
 \ee
 using  \eqref{P1-parti}.  
  Inserting the bounds \eqref{sum-1} and \eqref{sum-2}
 into \eqref{ES1y-a} completes the proof of \eqref{szyj}.  
\end{proof}

\begin{proof}[Proof of Theorem \ref{third} (ii)]
Let $H\in \mathfrak{H}$.  The case $j=0$ is trivial, so we turn attention to the $j=1$ claim:
\begin{equation}\label{j1}
 \E \sum_{q \in {\mathcal Q}_H} \sum_{\red{-Ky < n\le y}} \bm{\lambda}(H;q,n) = \left(1+O\left(\frac{1}{H^{M-2}}\right)\right) (K+1) y |{\mathcal Q}_H|.
\end{equation}
The left-hand expands as
$$
\E \sum_{q \in {\mathcal Q}_H} \sum_{\red{-Ky <n\le y}}
\frac{1_{\mathbf{AP}(KH;q,n) \subset \mathbf{S}_2}}{\sigma_2^{|\mathbf{AP}(KH;q,n)|}}.$$
Recalling the splitting \eqref{abbr-1} and \eqref{abbr-2},
that $\mathbf{b}_1$ and $\mathbf{b}_2$ are independent, and
consequently that $\mathbf{AP}(KH;q,n)$ and $\mathbf{S}_2$ are independent
(since the sets  $\mathbf{AP}(KH;q,n)$ defined in \eqref{APH} are determined by
$\mathbf{S}_1$).   The above expression then equals
\[
 \sum_{q \in {\mathcal Q}_H} \sum_{\red{-Ky < n\le y}}
 \sum_{b_1} \frac{\PR(\mathbf{b}_1=b_1) }{\sigma_2^{|\mathbf{AP}(KH;q,n)|}}  \PR(\mathbf{AP}(KH;q,n)\subset \mathbf{S}_2).
\]
Fix $\mathbf{b}_1$ and apply Lemma \ref{correlation}
with $\mathcal{U}=\mathbf{AP}(KH;q,n)$
and $\mathcal{V} = \{ n+ qh:1\le h\le KH\}$.
We find that the left side of \eqref{j1} equals
$$
 \sum_{q \in {\mathcal Q}_H} \sum_{\red{-Ky < n\le y}} \Bigg(1 + 
 O\pfrac{1}{H^{M-2}} + O\Bigg( \frac{1}{H^2} \ssum{1\le h,h' \leq KH \\ h \neq h'} E_{2BK^2H^2}( qh - qh' ;H) \Bigg) \Bigg).$$
Clearly it suffices to show that
$$
 \sum_{q \in {\mathcal Q}_H} E_{2BK^2H^2}( qh - qh' ;H) \ll \frac{|{\mathcal Q}_H|}{H^{M-2}}
$$
for any distinct $h,h'$ satisfying $1 \leq h,h' \leq KH$.  For future reference we will show the  more general estimate
\begin{equation}\label{eq-bound}
 \sum_{q \in {\mathcal Q}_H} E_{8B K^2 H^2}( q\ell + k;H) \ll \frac{|{\mathcal Q}_H|}{H^{M-2}}
\end{equation}
uniformly for any integer $k$ and $0<|\ell|\le KH$.
Note that $E_A(n;H)$ is increasing in $A$.

To prove \eqref{eq-bound}, fix $\ell,k$.  If $d \in {\mathcal D}_H \backslash \{1\}$ and $a\mod d$ is a residue class, all the prime divisors of $d$ are larger than $H^{M} > KH \geq |\ell|$; meanwhile, $q$ is larger than $z$ and is hence coprime to $d$.  Thus the relation $q\ell \equiv a\pmod{d}$ only holds for $q$ in at most one residue class modulo $d$, and hence by the Brun--Titchmarsh inequality we have
$$
\# \{ q \in {\mathcal Q}_H : q\ell \equiv a\pmod{d} \} \ll \frac{y/H}{\phi(d) \log y} $$
when (say) $d \leq \sqrt{y}$ (recall that $H\le (\log{y})^{1/2}$ by \eqref{H-bounds}).  For $d > \sqrt{y}$, we discard the requirement that $q$ be prime, and obtain the crude bound
$$
\# \{ q \in {\mathcal Q}_H : q\ell \equiv a\pmod{d} \} \ll
\frac{y/H}{d} + 1 \le  \frac{y/H}{\sqrt{y}}.$$
Thus for all $d$ we have
$$
\# \{ q \in {\mathcal Q}_H : q\ell\equiv a\pmod{d} \} \ll \frac{y}{H\phi(d) \log y} + \frac{\sqrt{y}}{H}
 $$
and hence by Lemma \ref{ds},
\begin{align*} 
\sum_{q \in {\mathcal Q}_H} E_{8B K^2 H^2}( q\ell + k;H)
&\ll \frac{y}{H\log y} \frac{H^2}{H^{M}} + \frac{\sqrt{y}}{H} \exp\left(O( H^2\log \log y) \right)\\
&\ll |\mathcal{Q}_H| H^{2-M} +\frac{\sqrt{y}}{H} \exp\left(O( H^2 \log \log y) \right).
\end{align*}
We note that the $O$-bound in the exponential depends on $B$ and $K$.
The claim \eqref{eq-bound} now follows from 
the upper bound in \eqref{H-bounds}, namely that
 $H\le (\log y)^{1/2} (\log\log y)^{-1}$, together
 with the bounds \eqref{QH} on $|\mathcal{Q}_H|$.
Incidentally, this is the only part of the proof that
requires the full strength of the upper bound in \eqref{H-bounds}, but it does however constrain the size of $z$.

Now we turn to the $j=2$ case of Theorem \ref{third}(ii), which is
$$ \E \sum_{q \in {\mathcal Q}_H} \bigg(\sum_{\red{-Ky < n\le y}} \bm{\lambda}(H;q,n)\bigg)^2 = \left(1+O\left(\frac{1}{H^{M-2}}\right)\right) (K+1)^2 y^2 |{\mathcal Q}_H|.$$
The left-hand side may be expanded as
$$
\E \sum_{q \in {\mathcal Q}_H} \;\; \sum_{\red{ -Ky <n_1,n_2 \le y}}
\frac{1_{\mathbf{AP}(KH;q,n_1) \cup \mathbf{AP}(KH;q,n_2) \subset
\mathbf{S}_2}}{\sigma_2^{|\mathbf{AP}(KH;q,n_1)|+|\mathbf{AP}(KH;q,n_2)|} }.
$$
Apply Lemma \ref{correlation}  with
\begin{align*}
\mathcal{U}&=\mathbf{AP}(KH;q,n_1) \cup \mathbf{AP}(KH;q,n_2),\\  
\mathcal{V}&=\{n_1+qh:1\le h\le KH\} \cup \red{ \{n_2+qh:1\le h\le KH \}},
\end{align*}
so that \red{$|\mathcal{V}|=\ell \ge \fl{KH}$}.
\red{When $n_1\not\equiv n_2\pmod{q}$, $\mathbf{AP}(KH;q,n_1)$ and
$\mathbf{AP}(KH;q,n_2)$ are disjoint.  There are $O(y^2/q)=O(yH)$ 
pairs $(n_1,n_2)$ with $n_1\equiv n_2 \pmod{q}$, and for each such pair,
$|\mathbf{AP}(KH;q,n_1)|+|\mathbf{AP}(KH;q,n_2)| \le |\mathcal{U}|+KH$.
We also have $\sigma_2^{-KH} \ll y^{o(1)}$.}
Noting that $\mathbf{S}_2$ is independent of both $\mathbf{AP}(KH;q,n_1)$ and $\mathbf{AP}(KH;q,n_2)$, we \red{see that the previous expectation
is $O(y^{1+o(1)}H|\mathcal{Q}_H|)$ plus} 
\begin{multline*}
\sum_{q \in {\mathcal Q}_H} \sum_{\red{-Ky < n_1,n_2 \le y}}
\Bigg(1 + O\pfrac{1}{H^{M-2}} 
+O\Bigg(\frac{1}{H^2} 
\sum_{h,h' \leq KH} \Biggl( 1_{h \neq h'}E_{8BK^2H^2}( qh - qh';H) + \\
+ 1_{n_1\ne n_2} E_{8BK^2H^2}( n_1 + qh - n_2 - qh';H ) \Biggr) \Bigg) \Bigg).
\end{multline*}
Using \eqref{eq-bound}, we obtain an acceptable main term and error terms
for everything except for the summands with $h=h'$.
For any fixed $n_2$, any $d \geq 1$ and $a\mod d$,
$$
\# \{ \red{-Ky < n_1 \le y} : n_1-n_2 \equiv a\pmod{d} \} \ll \frac{y}{d} + 1 $$
so by Lemma \ref{ds}, we have
\[
\sum_{\red{-Ky < n_1,n_2 \le y}} E_{8BK^2H^2}( n_1 - n_2;H) \ll
 y^2 \frac{H^2}{H^{M}} + y \exp\left(O( H^2 \log \log y) \right) \ll
\frac{y^2}{H^{M-2}},
\]
again using \eqref{H-bounds}.
This completes the proof of the $j=2$ case, and so we have established \eqref{ii}.
\end{proof}

\begin{proof}[Proof of  Theorem \ref{third}(iii).]
The $j=0$ case follows from the $j=1$ case of part (i) (that is, \eqref{szyj}), so we turn to the $j=1$ case, which is
$$
 \E \sum_{n \in \mathbf{S} \cap [1, y]} \sum_{q \in {\mathcal Q}_H} \sum_{h \leq KH} \bm{\lambda}(H; q, n-qh ) = \left(1+O\left(\frac{1}{H^{M-2}}\right)\right) |{\mathcal Q}_H| \red{\cdot \fl{K H}} \sigma_1 y.$$
It suffices to show that for each $h \leq KH$, one has
\be\label{Elamq}
 \E \sum_{n \in \mathbf{S} \cap [1, y]} \sum_{q \in {\mathcal Q}_H} \bm{\lambda}(H; q, n-qh ) = \left(1+O\left(\frac{1}{H^{M-2}}\right)\right) |{\mathcal Q}_H| \sigma_1 y.
 \ee
The left-hand side can be expanded as
$$
\E \sum_{n \in \mathbf{S} \cap [1, y]} \sum_{q \in {\mathcal Q}_H} \frac{1_{\mathbf{AP}(KH;q, n-qh) \subset \mathbf{S}_2}}{\sigma_2^{|\mathbf{AP}(KH;q, n-qh)|}}.$$
By \eqref{abbr-1}, the constraint $n \in \mathbf{S} \cap [1, y]$ implies that $n \in \mathbf{S}_1 \cap [1, y]$.  Conversely, if $n \in \mathbf{S}_1 \cap [1, y]$, then $n \in \mathbf{AP}(H;q, n-qh)$, and the condition $n \in \mathbf{S} $ is subsumed in the condition that $\mathbf{AP}(KH;q, n-qh) \subset \mathbf{S}_2$.  Thus we may replace the constraint $n \in \mathbf{S} \cap [1, y]$ here with $n \in \mathbf{S}_1 \cap [1, y]$ and rewrite the above expression as
$$
\E \sum_{n \in \mathbf{S}_1 \cap [1, y]} \sum_{q \in {\mathcal Q}_H} \frac{1_{\mathbf{AP}(KH;q, n-qh) \subset \mathbf{S}_2}}{\sigma_2^{|\mathbf{AP}(KH;q, n-qh)|} }.$$
Recall that $\mathbf{S}_2$ is independent of $\mathbf{S}_1$ and of $\mathbf{AP}(KH;q, n-qh)$.
Applying Lemma \ref{correlation} as before, we may write the left side of \eqref{Elamq} as 
$$ \E \sum_{n \in \mathbf{S}_1 \cap [1, y]} \sum_{q \in {\mathcal Q}_H}
\Bigg(1 + O\left(\frac{1}{H^{M-2}}\right) +
O\Bigg( \frac{1}{H^2} \ssum{h',h'' \leq KH \\ h' \neq h''} E_{8BK^2H^2}( qh' - qh'' )\Bigg)\Bigg).
$$
Trivially we have
\begin{equation}\label{swq}
\E |\mathbf{S}_1 \cap [1, y]| = 
\sum_{n=1}^y \PR(n\in \SS_1) = \sigma_1 y,
\end{equation}
and the claim \eqref{Elamq} now follows from \eqref{eq-bound}.

Finally, we establish the $j=2$ case of Theorem \ref{third}(iii), which expands as
\begin{multline*}
 \sum_{h_1,h_2 \leq KH} \E \sum_{n \in \mathbf{S} \cap [1, y]} \; \sum_{q_1,q_2 \in {\mathcal Q}_H} \;  \bm{\lambda}(H; q_1, n-q_1h_1 ) \bm{\lambda}(H; q_2, n-q_2h_2 ) = \\
 = \left(1+O\left(\frac{1}{H^{M-2}}\right)\right) |{\mathcal Q}_H|^2 \red{\fl{KH}^2} \frac{\sigma_1}{\sigma_2} y.
 \end{multline*}
\red{We can use \eqref{lambda-def} to expand the sum as}
\be\label{Elamq1q2}
 \red{\sum_{h_1,h_2\le KH}} \E \sum_{n \in \SS \cap [1, y]} \sum_{q_1,q_2 \in {\mathcal Q}_H}
\frac{1_{\mathbf{AP}(KH;q_1,n-q_1h_1) \cup \mathbf{AP}(KH;q_2,n-q_2h_2) \subset \mathbf{S}_2}}{\sigma_2^{|\mathbf{AP}(KH;q_1,n-q_1h_1)|+|\mathbf{AP}(KH;q_2,n-q_2h_2)|}}.
\ee
\red{Crudely, using \eqref{QH}, the terms with $q_1=q_2$ contribute 
\[
\ll H^2 y |\mathcal{Q}_H| \sigma_2^{-2KH} \ll |\mathcal{Q}_H|^2 y^{o(1)}.
\]
Now assume that $q_1\ne q_2$.}
As in the $j=1$ case, we may replace the constraint $n \in \mathbf{S} \cap [1, y]$ here with $n \in \mathbf{S}_1 \cap [1, y]$.  Next, we observe that the set
$$ \mathbf{AP}(KH;q_1,n-q_1h_1) \cup \mathbf{AP}(KH;q_2,n-q_2h_2) $$
contains \red{exactly} $|\mathbf{AP}(KH;q_1,n-q_1h_1)|+|\mathbf{AP}(KH;q_2,n-q_2h_2)|-1$ distinct elements, \red{$n$ being the unique common element of
$\mathbf{AP}(KH;q_1,n-q_1h_1$ and $\mathbf{AP}(KH;q_2,n-q_2h_2)$ (recall that $q_1,q_2$ are much larger than $KH$).}
  Thus if we apply Lemma \ref{correlation}
(noting that $\SS_2$ is independent of
$\SS_1, \mathbf{AP}(KH;q_1,n-q_1h_1)$ and $\mathbf{AP}(KH;q_2,n-q_2h_2)$) after eliminating the duplicate constraint, we may write \red{the terms in \eqref{Elamq1q2} with $q_1\ne q_2$} as
$$
\red{\fl{KH}^2} \sigma_2^{-1} \E \sum_{n \in \SS_1 \cap [1, y]} \;\;\red{\ssum{q_1,q_2 \in {\mathcal Q}_H \\ q_1\ne q_2}}
\left(1 + O\left(\frac{1}{H^{M-2}} + \frac{E'(q_1) + E'(q_2) + E''(q_1,q_2)}{H^2} \right) \right)
$$
where
$$ E'(q) \coloneqq \ssum{h,h' \leq KH\\ h \neq h'} E_{8BK^2H^2}( qh - qh' ; H )$$
and
$$ E''(q_1,q_2) \coloneqq \ssum{h'_1,h'_2 \leq K H \\ h_1 \neq h'_1, h_2 \neq h'_2} E_{8BK^2H^2}( q_1 h'_1 - q_1 h_1 - q_2 h'_2 + q_2 h_2;H ).$$
The average over $E'(q_1)+E'(q_2)$ is acceptably small by the $j=1$ analysis.  Thus (using \eqref{swq}) it suffices to show that
$$ \sum_{q_1,q_2 \in {\mathcal Q}_H} E_{8BK^2H^2}( q_1 h'_1 - q_1 h_1 - q_2 h'_2 + q_2 h_2 ;H)
\ll \frac{1}{H^{M-2}} |{\mathcal Q}_H|^2$$
for each $h'_1,h'_2 \leq KH$ with $h'_1 \neq h_1$, $h'_2 \neq h_2$).
But this follows from \eqref{eq-bound} (applied with $q$ replaced by $q_1$ and $k$ replaced by $-q_2 h'_2 + q_2 h_2$, and then summing in $q_2$). This completes the proof of the $j=2$ case, and so establishes \eqref{da}. 
\end{proof}
We have now verified all the the claims \eqref{ifirst}-\eqref{da}, and so have completed  the proof of Theorem \ref{third}.
\appendix

\section{Proof of the covering lemma}

In this appendix we prove Lemma \ref{hcl}.  Our main tool will be the following general hypergraph covering lemma from \cite[Theorem 3]{FGKMT}:

\newtheorem*{fgkmt-maintheorem}{Theorem A}
\begin{fgkmt-maintheorem}[Probabilistic covering]\label{packing-quant}  There exists
  an absolute constant $C_4 \geq 1$ such that the following holds.  Let $D, r, A
  \geq 1$, $0 < \kappa \leq 1/2$, and let $m \geq 0$ be an integer.
  Let $\tau > 0$ satisfy 
\begin{equation}\label{delta-small}
\tau \leq \left( \frac{\kappa^A}{C_4 \exp(AD)} \right)^{10^{m+2}}.
\end{equation}
Let $I_1,\dots,I_m$ be disjoint finite non-empty sets, and let $V$ be a finite set.  For each $1 \leq j \leq m$ and $i \in I_j$, let $\mathbf{e}_i$ be a random  subset of $V$.  Assume the following:
\begin{itemize}
\item (Edges not too large) Almost surely for all $j=1,\dots,m$ and $i \in I_j$, we have
\begin{equation}\label{edge-bound}
 \# \mathbf{e}_i \leq r;
\end{equation}
\item (Each sieve step is sparse) For all $j=1,\dots,m$, $i \in I_j$ and $v \in V$,
\begin{equation}\label{v-sparse}
\PR( v \in \mathbf{e}_i ) \leq \frac{\tau}{|I_j|^{1/2}};
\end{equation}
\item (Very small codegrees) For every $j=1,\dots,m$, and distinct $v_1,v_2 \in V$,
\begin{equation}\label{codegree-a}
\sum_{i \in I_j} \PR( v_1,v_2 \in \mathbf{e}_i ) \leq \tau
\end{equation}
\item (Degree bound) If for every $v \in V$ and $j =1,\dots,m$ we introduce the normalized degrees
\begin{equation}\label{epsqj-quant}
 d_{I_j}(v) := \sum_{i \in I_j} \PR( v \in \mathbf{e}_i )
\end{equation}
and then recursively define the quantities $P_j(v)$ for $j=0,\dots,m$ and $v \in V$ by setting
\begin{equation}\label{p0-def}
P_0(v) := 1
\end{equation}
and
\begin{equation}\label{pj-def}
P_{j+1}(v) := P_j(v) \exp( - d_{I_{j+1}}(v) / P_j(v) )
\end{equation}
for $j=0,\dots,m-1$ and $v \in V$, then we have
\[
d_{I_j}(v) \leq D P_{j-1}(v) \qquad (1 \le j\le m, v\in V)
\]
 and
\[
P_j(v) \geq \kappa \qquad (0 \le j\le m, v\in V).
\]
\end{itemize}
Then there are random variables $\mathbf{e}'_i$ for each $i \in \bigcup_{j=1}^m I_j$ with the following properties:
\begin{itemize}
\item[(a)]  For each $i \in \bigcup_{j=1}^m I_j$, the support of $\mathbf{e}'_i$ is contained in the support of $\mathbf{e}_i$, union the empty set singleton $\{\emptyset\}$.  In other words, almost surely $\mathbf{e}'_i$ is either empty, or is a set that $\mathbf{e}_i$ also attains with positive probability.
\item[(b)]  For any $0 \leq J \leq m$ and any finite subset $e$ of $V$ with $\# e \leq A - 2rJ$, one has
\[
 \PR\left( e \subset V \backslash \bigcup_{j=1}^J \bigcup_{i \in I_j} \mathbf{e}'_i \right)
 = \(1 + O( \tau^{1/10^{J+1}} )\) P_J(e)
\]
where
\[
P_j(e) := \prod_{v \in e} P_j(v).
\]

\end{itemize}
\end{fgkmt-maintheorem}

\begin{proof} See \cite[Theorem 3]{FGKMT}.
\end{proof}

To derive Lemma \ref{hcl} from Theorem \ref{packing-quant}, we repeat the proof of \cite[Corollary 4]{FGKMT} with a different choice of parameters.  Let the notation and hypotheses be as in Lemma \ref{hcl}. 
Firstly, we may assume that $\eta \le \frac{1}{1000}$, 
for the conclusion is trivial otherwise.

Let $\beta=\beta(\delta)$ be a parameter satisfying
\be\label{beta}
\beta > 10^{2\delta} > \frac{\beta \log \beta}{\beta-1}
\ee
This is possible as $\log \beta < \beta-1$ for all $\beta>1$.
Let
\be\label{mdef}
m = \cl{\frac{\log (1/\eta)}{\log \beta}}
\ee
so that, by \eqref{c2-bound},
\be\label{m-bounds}
1\le m\le \frac{\delta \log\log y+\log\log\log y}{\log \beta}+1, \qquad
\frac{1}{\eta} \le \beta^m \le \frac{\beta}{\eta}.
\ee
 By \eqref{c2-bound} and \eqref{beta},
  $C_2 > \frac{\beta \log \beta}{\beta-1}$
 and thus we may find disjoint intervals $\mathscr{I}_1,\dots,\mathscr{I}_m$ in $[0,1]$ with length
\begin{equation}\label{ij}
|\mathscr{I}_j| = \frac{\beta^{1-j}\log \beta}{C_2}
\qquad (1\le j\le m).
\end{equation}
Let $\vec{\mathbf{t}} = (\mathbf{t}_1,\ldots,\mathbf{t}_s)$,
where $\mathbf{t}_i$ is a uniform random real number in $[0,1]$ for each $i$, and such that $\mathbf{t}_1,\dots,\mathbf{t}_s$ are independent.  Define the random sets
$$ I_j = I_j( \vec{\mathbf{t}} ) := \{ 1\le i\le s: \mathbf{t}_i \in \mathscr{I}_j \}$$
for $j=1,\dots,m$.  These sets are clearly disjoint.

We will verify (for a suitable choice of $\vec{\mathbf{t}}$) the hypotheses of Theorem \ref{packing-quant} with the indicated sets $I_j$ and random variables $\mathbf{e}_i$, and with suitable choices of parameters $D, r, A \geq 1$ and $0 < \kappa \leq 1/2$. 

Let $v \in V$, $1\le j\le m$ and consider the independent random
variables $(\mathbf{X}_i^{(v,j)}(\vec{\mathbf{t}}))_{1\le i\le s}$, where
$$
\mathbf{X}_i^{(v,j)}(\vec{\mathbf{t}})=\begin{cases}
\PR( v \in \mathbf{e}_i) & \text{ if } i\in I_j(\vec{\mathbf{t}}) \\
0 & \text{ otherwise.}
\end{cases}
$$
By \eqref{uniform}, \eqref{ij}, and \eqref{m-bounds}, we have for every $1 \leq j \leq m$ and $v \in V$ that
\begin{align*}
\sum_{i=1}^s \E \mathbf{X}_i^{(v,j)}(\vec{\mathbf{t}}) &=
\sum_{i=1}^s \PR( v \in \mathbf{e}_i) \PR(i\in I_j(\vec{\mathbf{t}}))
\\ &= |\mathscr{I}_j| \sum_{i=1}^s
\PR( v \in \mathbf{e}_i) \\
&= \beta^{1-j}\log \beta + O \big( \eta \beta^{-j}\log \beta  \big)\\
&= \beta^{1-j}\log \beta +O \big(  \beta^{-m-j}\log\beta \big).
\end{align*}
In the last equality we have used that $C_2 \ge 1$.

By \eqref{sparsity}, we have
$|\mathbf{X}_i^{(v,j)}(\vec{\mathbf{t}})| \le y^{-1/2-1/100}$ for all $i$, and hence by
Hoeffding's inequality,
\begin{align*}
\PR\( \Big| \sum_{i=1}^s (\mathbf{X}_i^{(v,j)}(\vec{\mathbf{t}})-\E \mathbf{X}_i^{(v,j)}(\vec{\mathbf{t}})) \Big| \ge
\frac{1}{y^{1/200}} \) &\le 2 \exp \left\{ - 2\frac{y^{-1/100}}{y^{-1-1/50} s} \right\}  \\
&= 2 \exp \left\{ - 2y^{1/100} \right\}.
\end{align*}
Here we used the hypothesis $s\le y$.
By a union bound, the bound $|V|\le y$ and \eqref{mdef}, there is a deterministic
choice $\vec{t}$ of $\vec{\mathbf{t}}$ (and hence $I_1,\dots,I_m$)
such that for \emph{every} $v \in V$ and \emph{every}
$j=1,\ldots,m$, we have
\[
\Big| \sum_{i=1}^s (\mathbf{X}_i^{(v,j)}(\vec{t})-\E \mathbf{X}_i^{(v,j)}(\vec{\mathbf{t}})) \Big| < \frac{1}{y^{1/200}}.
\]
Note that this is vastly smaller than $\beta^{-m}\order (\log y)^{-\delta}$.
We fix this choice $\vec{t}$ (so that the $I_j$ are now deterministic),
and we conclude that for $y$ sufficiently large (in terms of $\delta$)
\begin{equation}\label{oo}
\begin{split}
\sum_{i \in I_j} \PR( v \in \mathbf{e}_i) &=\sum_{i=1}^s \mathbf{X}_i^{(v,j)}(\vec{t}) \\
&=  \beta^{1-j}\log \beta + O\( \beta^{-j-m}\log\beta +\frac{1}{y^{1/200}}\)
\\ &=  \beta^{1-j}\log \beta + O\( \beta^{-j-m}\log\beta \)
\end{split}
\end{equation}
uniformly for all $j=1,\dots,m$, and all $v \in V$.
In particular, all sets $I_j$ are nonempty.

Set
\begin{equation}\label{deltaa}
 \tau \coloneqq y^{-1/100}
\end{equation}
and observe from \eqref{sparsity} and the bound $|I_j| \leq s \leq y$ that the sparsity condition \eqref{v-sparse} holds.  Also, the small
codegree condition \eqref{codegree} implies the small codegree condition
\eqref{codegree-a}.

From \eqref{epsqj-quant}, \eqref{oo} and \eqref{m-bounds}, we now have
$$ d_{I_j}(v) = (1 + O(\beta^{-m})) \beta^{-j+1} \log \beta$$
for all $v \in V$, $1\le j\le m$.
Let $\lambda$ satisfy $1+\log\beta < \lambda < \beta$.
A routine induction using \eqref{p0-def}, \eqref{pj-def} then shows (for $y$ sufficiently large) that
\begin{equation}\label{moldy}
 P_j(v) = (1 + O( \lambda^j \beta^{-m} ) ) \beta^{-j} \quad (0\le j\le m),
\end{equation}
In particular we have
$$ d_{I_j}(v) \leq D P_{j-1}(v)\qquad (1\le j\le m)$$
for some absolute constant $D$, and
$$ P_j(v) \geq \kappa \qquad (0\le j\le m),$$
where
$$ \kappa \gg \beta^{-m} \ge \eta/\beta \gg \eta.$$
We now set
$$ \red{r= \frac{K(\log y)^{1/2}}{\log\log y}}, \qquad  A := 2rm+1.$$
By \eqref{m-bounds} and \eqref{size-bound}, one has
$$ A \ll (\log y)^{1/2}$$
 and so \eqref{edge-bound} holds and also
\be\label{KACD}
 \frac{\kappa^A}{C_4 \exp(AD)} \gg \exp\( - O\( (\log y)^{1/2} (\log\log y) \) \). 
\ee
By \eqref{mdef} and \eqref{beta}, 
\[
10^m \ll (1/\eta)^{\frac{\log 10}{\log \beta}} \ll
(\log y)^{\frac{\delta \log 10}{\log \beta}} (\log\log y)^{\frac{\log 10}{\log \beta}} <
(\log y)^{1/2-\eps_1}
\]
for some $\eps_1=\eps_1(\delta)>0$.
Hence by \eqref{deltaa},  we see that
\be\label{deltam10m}
\tau^{1/10^{m+2}} \le \exp \Bigg\{ - K(\log y)^{1/2+\eps_1} \Bigg\},
\ee
for some absolute constant $K>0$.
Combining \eqref{KACD} and \eqref{deltam10m}, we see that  \eqref{delta-small} is satisfied if $y$ is large enough. Thus all the hypotheses of Theorem \ref{packing-quant} have been verified for this choice of parameters. Applying this Theorem \ref{packing-quant} and using \eqref{moldy}, one thus obtains random variables $\mathbf{e}'_i$ for $i \in \bigcup_{j=1}^m I_j$ whose range is contained in the range of $\mathbf{e}_i$ together with $\emptyset$, such that
\[
 \PR\left( n \not \in \bigcup_{j=1}^m \bigcup_{i \in I_j} \mathbf{e}'_i \right) \ll \beta^{-m} \ll \eta
\]
for all $n \in V$. 
For $1\le i\le s$, $i\not \in \bigcup_{j=1}^m I_j$, set 
$\mathbf{e}_i' = \emptyset$ with probability 1.
 By linearity of expectation this gives 
\[
\E \Big|V \setminus \bigcup_{i=1}^s \mathbf{e}_i'  \Big|
\ll \eta |V|.
\]
Hence, for some absolute constant $C_3>0$, we have
\[
\Big|V \setminus \bigcup_{i=1}^s \mathbf{e}_i'  \Big|
\le C_3 \eta |V|
\]
with probability $\ge 1/2$.
Therefore, there is some vector $(e_1,\ldots,e_s)$ of
subsets of $V$, where, for every $i$, $e_i$ is in the support of  $\mathbf{e}_i$ or is the empty set, for which
\eqref{vbound} holds.  Finally, for the $i$ such that
$e_i$ is the empty set, replace $e_i$ with an arbitrary
element in the support of $\mathbf{e}_i$; clearly 
\eqref{vbound} still holds.


\appendix

\section{Corrigendum: changes made from the published version}

This document incorporates a number of corrections to the published version
of the paper, JEMS {\bf 23} (2021), 667--700.
The authors are grateful to Mikhail Gabdullin for pointing these out to us.

The only error which affect the results of the paper are
are errors in the exponents of $H$ in the deduction of Theorem 2
from Theorem 3.
When corrected, these force the parameter $M$ to be somewhat larger than claimed, namely  $M>6$.  This affects the numerical estimates 
for the exponents of $\log\log x$ in Theorem 1 and corollaries.

 Below we enumerate the specific corrections to
 the published version, which are all incorporated in the present document.
 The page number(s) in parentheses refer to the published version.

\begin{enumerate}

\item (p. 669) In Theorem 1, the definition of $C(\rho)$, the factor $4+\delta$ corrected to 6.  Likewise, the corrected lower bound is $C(\rho)>e^{-1-6/\rho}$.
Corrected (2.3) and the following display accordingly.  The corrected asymptotic, five lines after (2.3), is $C(\rho)\sim \frac12 e^{-6/\rho}$ as $\rho\to 0^+$.

\item (p. 669) In Example 1, the corrected bound is $C(1) > 1/835$.

\item (p. 670) In Corollary 1, the corrected lower bound is $C(1/d) > e^{-(6d+1)}$.

\item (p. 671) In (1.7) and Corollary 2, the corrected bound is $C(1/2) > 1/325565$.

\item (p. 675) Six lines after (2.3), we state that $M$ is a fixed number slightly larger than 6.

\item (p. 678): In (2.10), we write $6 <M \le 7$.  Three lines before Remark 9, we write ``$M$ sufficiently close to 6''.

\item (p. 680): The hypotheses of Theorem 2 are adjusted slightly.  With $\delta$ fixed
satisfying (2.3), $M$ is taken sufficiently close to 6, $\xi$ sufficiently close
to $1$, $K$ sufficiently large (all depending on $\delta$) and $\eps$ satisfying $M < 6+6\eps$, with $x$ 
sufficiently large in terms of all of these parameters.

\item (p. 682) in the third to last display in section 3, a missing factor of $K$ is added all terms, and it now reads
\[
|\mathbf{e}_q| \le K H_q \le \frac{Ky}{z} = \frac{K(\log x)^{1/2}}{\log\log x} \le \frac{K(\log y)^{1/2}}{\log\log y}.
\] 
Consequently, we add a factor $K$ to the right side of (3.5), stipulate in
Lemma 3.1 that $y \ge y_0(\delta,K)$ with $y_0(\delta,K)$ sufficiently large,
and add a factor of $K$ to the definition of $r$, four lines before (A.15).

\item (p. 683) In the final two-line display of section 3, we correct the conclusion
to \[
C_2 + O((\log x)^{-\delta(1+\eps)}).
\]

\item (p. 683) In (4.4), the factor $KH$ on the right side is corrected to $\fl{KH}$, since
$KH$ need not be an integer.  This induces other changes: we change $|\mathcal{Q}_H|KH$ to $|\mathcal{Q}_H|\cdot \fl{KH}$ twice in the display preceding (4.10), twice in (4.10), on the right side of (4.11), and in the fourth display following (4.11).  We change the definition of $C_2$ (toward the end of section 4) to 
\[
C_2 = \frac{1}{(K+1)y} \sum_{H\in \mathfrak{H}} \frac{|\mathcal{Q}_H|\cdot \fl{KH}}{\sigma_2}.
\]
We correct $KH$ to $\fl{KH}$ on the right side of the display before (5.9),
and correct $K^2H^2$ to $\fl{KH}^2$ on the right side of the display before (5.11).

\item (p. 685) In (4.10), the denominator on the right side is corrected to $\sigma_2 H^{1+\eps}$. The four lines following (4.10) are corrected as follows: 
``then, recalling that $M>6$ and $\eps$ is very small,
\[
\E |\mathcal{E}_H| \ll \frac{\sigma y}{H^{1+2\eps}}.
\]
By Markov's inequality, we conclude that $|\mathcal{E}_H| \le \sigma y/H^{1+\eps}$ with probability $1-O(H^{-\eps})$.''
\smallskip

\item (p. 685) In three places in between (4.10) and (4.11), the summation $\ds \sum_{n=1}^y$ is corrected to $\ds \sum_{-Ky < n-hq \le y}$.

\item (p. 685--86) We correct the denominator on the right side of (4.11) to $H^{1+\eps}\sigma_2$.
The following lines are then corrected as:
``Then
\[
\mathbb{E} |\mathcal{E}_H'| \ll \frac{y H^{1+\eps} \sigma_2}{H^{M-4-2\eps}} \ll \sigma y \frac{\log H}{H^{M-5-3\eps}}.
\]
By Markov's inequality, $|\mathcal{E}_H'| \le \sigma y/H^{1+\eps}$ with
probability $1-O(1/H^{M-6-5\eps})$.
By (2.10) again, if
$\eps$ is small enough then $M-6-5\eps>\eps$.
Consider the event that (4.5) holds,
and that for every $H$, we have
(4.9),  $|\mathcal{E}_H| \le \sigma y/H^{1+\eps}$ and
 $|\mathcal{E}_H'| \le \sigma y/H^{1+\eps}$.''
 
\medskip

\item (p. 686) in the fourth display after (4.11),  the big-$O$ term is corrected to $\ds O\bigg(\frac{1}{H^{1+\eps}}\bigg)$.

\item (p. 688)  line -7.  In the proof of Lemma 5.1, the inequality $k\le 10H$ is corrected to $k\le 10KH$.

\item (p. 690) In the display prior to (5.4), we added a missing error term. The line now reads
\[
\mathbb{P}( n_1,n_2 \in \mathbf{S}_2 ) = \left( 1 + O(H^{-M})+ O\left(E_{8B}(n_1-n_2;H)\right)\right) \sigma_2^2.
\]
Consequently, we added a factor $(1+O(1/\log y))$ at the beginning of the second line
of (5.4).

\item (p. 691) Two lines after (5.5), the relation $d\in \mathcal{D}_{H^+}$
is corrected to $d\in \mathcal{D}_{H} \setminus \{1\}$.

\item (p. 691) We added a missing factor $P_1^2$ to the denominator in the final fraction in (5.6).

\item (p. 691--93)  In several places, we wrote that variables are $\ge -Ky$ and it is corrected to $>-Ky$.  This correction is made four lines after the statement of Theorem 3 and in seven places in section 5.

\item (p. 693) In the definition of $\mathcal{V}$, midway between (5.8) and (5.9), 
we corrected $1\l H\le KH$ to $1\le h \le KH$.

\item (p. 693) in the proof of the $j=2$ case of Theorem 3 (ii),
the argument as written works unless $n_1\equiv n_2 \pmod{q}$.
To take this case into account, replace the two lines following the 
definition of $\mathcal{V}$ with the following:
``so that $|\mathcal{V}|=\ell \ge \fl{KH}$.
When $n_1\not\equiv n_2\pmod{q}$, $\mathbf{AP}(KH;q,n_1)$ and
$\mathbf{AP}(KH;q,n_2)$ are disjoint.  There are $O(y^2/q)=O(yH)$ 
pairs $(n_1,n_2)$ with $n_1\equiv n_2 \pmod{q}$, and for each such pair,
$|\mathbf{AP}(KH;q,n_1)|+|\mathbf{AP}(KH;q,n_2)| \le |\mathcal{U}|+KH$.
We also have $\sigma_2^{-KH} \ll y^{o(1)}$.
Noting that $\mathbf{S}_2$ is independent of both $\mathbf{AP}(KH;q,n_1)$ and $\mathbf{AP}(KH;q,n_2)$, we see that the previous expectation
is $O(y^{1+o(1)}H|\mathcal{Q}_H|)$ plus''

\medskip

\item (p. 694--95) In the proof of Theorem 3 (iii), $j=2$ case, 
the case $q_1=q_2$ requires special analysis, which has now been included.
By (2.8) these terms contribute $\ll H^2 y |\mathcal{Q}_H| \sigma_2^{-2KH} 
\ll |\mathcal{Q}_H|^2 y^{o(1)}$, which is negligible.
Also, ``we may write (5.11) as'' is changed to
``we may write the sum of (5.11) over $h_1,h_2$ as'', and we added 
a factor $\fl{KH}^2$ to the following display.
The reason for this change is that $E'(q_1)$, $E'(q_2)$
and $E''(q_1,q_2)$ already incorporate sums over $h_1,h_2$.

\end{enumerate}

\end{document}